\documentclass[reqno,11pt]{amsart}
\usepackage{amsmath, amssymb, amsthm,amsfonts,pgfplotstable} 
\usepackage[english]{babel}
\usepackage{bbm}
\usepackage{graphicx}
\usepackage{url}
\usepackage{pgfplots} \pgfplotsset{compat=newest}
\usepgfplotslibrary{groupplots}
\usepackage{epstopdf}
\usepackage[a4paper,bindingoffset=0.5cm,left=2.5cm,right=2.5cm,top=2.5cm,bottom=2cm,footskip=.8cm]{geometry}

\usepackage{rotating}
\usepackage{amsbsy,enumerate}

\usepackage{comment}
\usepackage{mathrsfs}

\usepgfplotslibrary{groupplots}
\usepackage{tikz}
\usepackage{rotating}
\usepackage{amsbsy,enumerate}
\usepackage{graphicx}
\usepackage{comment}
\usepackage{mathrsfs}

\newtheorem{lemma}{Lemma}
\newtheorem{corollary}{Corollary}
\newtheorem{assumption}{Assumption}
\newtheorem{theorem}{Theorem}
\newtheorem{remark}{Remark}

\newtheorem{definition}{Definition}
\newtheorem{proposition}{Proposition}

\allowdisplaybreaks

\usepackage[utf8]{inputenc}
\usepackage{amsmath,amssymb,amsfonts}
\usepackage{amsthm}
\usepackage{array} % for better looking arrays
\usepackage{url,hyperref}
\usepackage{graphicx}
\usepackage{epstopdf}
\usepackage{color}
\usepackage{pgfplotstable}
\usepackage{siunitx}
\usepackage{booktabs}
\usepackage{caption}
\usepackage{empheq}
\usepackage{verbatim}
\usepackage{enumitem}

\renewenvironment{itemize} % removes skips between itemsimportance sampling
        {\begin{list}
                {$\bullet$}{\setlength{\parskip}{0cm} \setlength{\topsep}{0cm}
                 \setlength{\partopsep}{0cm} \setlength{\itemsep}{0cm} \setlength{\parsep}{0cm}\item[]}}
        {\end{list}}

\theoremstyle{plain}

\newcommand{\nneg}{\mathbb{R}_+}

\newcommand{\zplus}{\mathbb{Z}_+}

\newcommand{\pr}{\mathbb{P}}
\newcommand{\ex}{\mathbb{E}}
\newcommand{\ind}{\mathbb{I}}

\newcommand{\eps}{\varepsilon}
\newcommand{\dlt}{\delta}
\newcommand{\lmb}{\lambda}

\newcommand{\lf}{\lfloor}
\newcommand{\rf}{\rfloor}
\newcommand{\lc}{\lceil}
\newcommand{\rc}{\rceil}

\newcommand{\wtil}{\widetilde}

\newcommand{\dcopy}{\mathop{\rm =}\limits^{\text{d}}} %DO I NEED THIS
\begin{document}

%\title{Random fluid limit for bitcoin blockchain dynamics}

\title[A Bitcoin-inspired infinite-server model 
with a random fluid limit]{A Bitcoin-inspired infinite-server model \\
with a random fluid limit}
\author{Maria Frolkova \& Michel Mandjes}
\date{\today}

\maketitle

\begin{abstract}
The synchronization process inherent to the Bitcoin network gives rise to an infinite-server model with the unusual feature that customers interact. Among the closed-form characteristics that we derive for this model is the busy period distribution which, counterintuitively, does not depend on the arrival rate. We explain this by exploiting the equivalence between two specific service disciplines, which is  also used to derive the model's stationary distribution.
% (in a generalized sense, since the model is non-Markov). 
Next to these closed-form results, the second major contribution concerns an asymptotic result: a fluid limit in the presence of service delays. Since fluid limits arise under scalings of the law-of-large-numbers type, they are usually deterministic, but in the setting of the model discussed in this paper the fluid limit is random (more specifically, of growth-collapse type). 

\vspace{2mm}


\vspace{2mm}

\noindent {\sc Affiliations.} M. Frolkova and M. Mandjes are with Korteweg-de Vries Institute for Mathematics, University of Amsterdam, Science Park 105-107, 1098 XG Amsterdam, the Netherlands. 
M. Frolkova is corresponding author (email: {\scriptsize\tt maria.frolkova@gmail.com}). 
\end{abstract}

\section{Introduction} 
The {\it Bitcoin network} is a payment system where all transactions between participants are carried out in the digital Bitcoin currency. These transactions are stored in a database called {\it blockchain}. In fact, every Bitcoin user maintains a blockchain version, thus keeping track of the global Bitcoin transaction history. This ensures transparency --- one of the advertized Bitcoin features. Information spreads across the network via peer-to-peer communication. All transactions get broadcast, and certain users arrange them into blocks suitable to add to the blockchain. The blocks get broadcast, too, and other users update their blockchain versions as soon as they hear of new blocks. It is vital for the network to synchronize
%, or agree on every block in the blockchain, 
as soon as possible because the Bitcoin payment system is trustworthy to the degree it is synchronized. This explains the need for models that capture 
the desyncronization caused by communication delays. Due to such delays, some Bitcoin users are not up to date compared to others, meaning that their blockchain versions lack some blocks compared to the blockchain versions of others. This desynchronization effect motivates the model analyzed in the present paper.
%The model analyzed in the present paper describes this desyncronization effect. 
It can be seen as a seemingly minor modification of a conventional queueing system, but this modification gives rise to various unexpected features and, overall, intrinsically different behavior.

In our model we assume Bitcoin-like communication between two parties, where party~A generates pieces of information and sends them off to party B, and B receives these information pieces with delay, i.e.\ B is not up to date compared to A. Such communication can be described by an infinite-server model where newly generated information pieces correspond to arrivals and their transmission corresponds to receiving service at parallel servers. An unconventional feature of this model is that  customers do interact here (usually, in infinite-server models, customers do not interfere and act independently, which often facilitates explicit analysis). This  interaction is inspired by the following mechanism in the blockchain dynamics. In the blockchain, each next block contains a reference to the previous one. Suppose participant B receives a block from participant A who knows more than B, and this new block does not directly link to B's last block. Then B requests from A the new block's predecessors using the references to the previous blocks. Once all the missing blocks have been retrieved, B adds the entire batch to his blockchain. On a somewhat stylized level, such update in batches translates  to the following service discipline: whenever there is a service completion, then {\it all earlier arrivals} that are still in the system complete service and depart as well. That is, departures occur in batches, which we refer to as FIFO-batches. This departure mechanism leads to crucial differences compared to the conventional infinite-server model where customers depart one by one upon their own service completion.

In the above model, we assume a renewal arrival process and exponentially distributed service times. Our first contribution concerns a number of closed-form characteristics, in particular for the busy period distribution. Remarkably, this distribution turns out to {\it not} depend on the inter-arrival distribution or the arrival rate, contradicting the intuition that a large arrival rates  should result in a longer busy period. This seeming paradox is resolved relying on an equivalent service discipline --- LIFO-batch departures. Distributionally, the model's evolution is the same under either discipline. This equivalence also makes it possible to identify the corresponding stationary distribution (in the generalized sense, since the underlying process is not Markov due to renewal arrivals). It is unclear whether this result could be obtained analytically under the original FIFO-batch departures discipline. Other closed-form results are obtained under the additional assumption of Poisson arrivals (such that the queue-length process is a continuous-time Markov chain): by direct methods, we derive moments of the stationary distribution and the queue-size distribution at any given time instant.

The second strand of results that we obtain is of an asymptotic nature. We develop a fluid limit for the queue-length process in presence of service delays. As fluid limits arise under law-of-large-numbers scalings, they are typically deterministic; this is for instance the case for the conventional infinite-server model. In our model, however, the fluid limit is radically different: rather than being deterministic, in the fluid limit some randomness is retained. More specifically, the fluid limit corresponds to a growth-collapse process: it grows linearly between randomly occurring downward jumps, where the jump rate is proportional to the current position. At each jump, the state undergoes a uniform cut-off (that is, the position drops to a level that is uniformly distributed between 0 and the pre-jump position). Growth-collapse processes have been studied in the literature and are used to describe a wide range of phenomena ranging from seismic to neuronal activity; see e.g.~\cite{OnnoGC, Last}  and references therein.

Interestingly, growth-collapse processes similar to the the fluid limit in our paper arise as scaling limits of Markovian AIMD (additive-increase/multiplicative-decrease) congestion control schemes analysed by Dumas et al.\ \cite{AIMD1} and  Guillemin et al.\ \cite{AIMD2}. However, the generator approach that quickly yields the fluid limits in \cite{AIMD1,AIMD2}  is non-applicable to our non-Markovian model. In this regard, this paper suggests a non-Markovian counterpart for the fluid limit analysis in \cite{AIMD1,AIMD2}. We present the proof of the fluid limit theorem under FIFO-/LIFO-batch departures, which correspond to the uniform cut-offs in the fluid limit. However, the proof is sufficiently flexible to allow for more general service disciplines that  result in generally distributed multiplicative cut-offs in the fluid limit. 

In more detail, the following techniques are used in the proof of the fluid-limit theorem. First, we identify an embedded Markov chain that converges to the post- and pre-jump fluid limit levels. This part of the proof builds on the embedded Markov chain analysis in \cite{AIMD1,AIMD2}. The derivation of the Weibullian increments between the jumps in the fluid limit (which yield the jump rate proportional to the state)  is similar to the corresponding derivation in  \cite{AIMD1,AIMD2}. The derivation of the uniform cut-offs in the fluid limit is more involved than the corresponding conclusion in \cite{AIMD1,AIMD2}, where the multiplicative-decrease factor is fixed throughout the pre-limit models and is directly inherited by the fluid limit. In our case, the downward jumps of the embedded Markov chain only become uniform in the limit, which we additionally prove. In the second part of the proof, we establish convergence of the continuous-time  queue-length process based  on the convergence of the embedded Markov chain. We follow the conventional approach: first we show convergence of finite-dimensional distributions and then verify a compact-containment and oscillation-control conditions, which imply the relative-compactness of the pre-limit processes. As a final note on the proof techniques, we emphasize the following methodological difference with earlier work. In \cite{AIMD1,AIMD2}, the embedded Markov chain is used to find (the fluid limit of) the stationary distribution of the full Markov chain. The fluid limit of the full Markov chain itself is obtained via generators. In this paper, (the fluid limit of) the stationary distribution of the queue-length process is known directly. We use the embedded Markov chain to derive the fluid limit of the non-Markov queue-length process.

Besides the mentioned connection with the AIMD and growth-collapse literature, our paper also relates, on the modelling level,  to the Bitcoin study by G\"{o}bel et al.~\cite{PTaylor} on whether it could be profitable to abuse the Bitcoin protocol by hiding information instead of propagating it. The model that is the starting point for the analysis in~\cite{PTaylor} and our model are complementary in the following sense. G\"{o}bel et al.~\cite{PTaylor} consider Bitcoin-like communication between two parties that both generate new information and communicate it to each other, while we consider one-sided communication. On the other hand, \cite{PTaylor} omits some communication transitions to ensure analytical tractability of the model, whereas those are taken into account in the present paper. More specifically, \cite{PTaylor} assumes complete synchronization whenever communication takes place. However, unless the latest information has been communicated, the two parties only partially synchronize, i.e., remain desynchronized; we succeed to include this aspect in our model. %While the models in this paper and in~\cite{PTaylor} generalize each other in different aspects, neither of the analyses allows for easy incorporation of the omitted transitions. 

The rest of the paper is organized as follows. In Section~\ref{sec:stationary_analysis}, we introduce the model, discuss the insensitivity of the busy period distribution with respect to the arrival rate, present formulae for the stationary distribution and its moments (under Poisson arrivals). In Section~\ref{sec:transient_analysis}, we characterize the transient behavior: give a formula for the queue-length distribution at a fixed time (under Poisson arrivals) and present the fluid limit theorem. In Section~\ref{sec:proof_stationary_analysis}, we provide the proofs for Section~\ref{sec:stationary_analysis} (noting that some core ideas have already been discussed in Section~\ref{sec:stationary_analysis}). In Section~\ref{sec:proof_FL}, we prove the fluid limit theorem. In the appendix, we prove the results that assume Poisson arrivals. In the remainder of this section, we list the notation used throughout the paper. 

{\bf Notation}. To define $x$ as equal to $y$, we write $x:=y$ or $y=:x$. The set of non-negative integers is $\mathbb{Z}_+ := \{0,1,2,\ldots\}$, and the non-negative half-line is $\mathbb{R}_+ := [0,\infty)$. The space of functions $f \colon \nneg \to \nneg$  that are right-continuous with left limits is denoted by $\mathbf{D}$, it is endowed with the Skorokhod $J_1$-topology. For functions $f(x)$, $g(x)$ that are asymptotically equivalent as $x \to \infty$, i.e.\ such that $f(x)/g(x) \to 1$ as $x \to \infty$, we write $f(x) \sim g(x)$. All random elements in the paper are defined on the same probability space with probability measure $\mathbb{P}$ and expectation operator $\mathbb{E}$. The signs $\Rightarrow$ and $\dcopy$ stand for weak convergence and equality in distribution, respectively. The complement of an event $B$ is denoted by $\overline{B}$. Finally, the following types of notation are used for vectors: $(d_l )_{l =1}^n := (d_1, d_2, \ldots, d_n)$ and $(X_{l -1}, Y_l )_{l =1}^n := (X_0,Y_1,X_1,Y_2, \ldots, X_{n-1},Y_n)$.

\section{Stochastic model and its stationary analysis} \label{sec:stationary_analysis}
In this section, we introduce the model under study and discuss a number of its properties. First, we argue that two different service disciplines are equivalent. This equivalence provides insights into the results that we discuss next: an insensitivity property of the busy period and stationary characteristics of the model.

\subsection{Model description} \label{sec:coupling}
The model we analyze is a variant of one of the classical queueing models --- the $G/M/\infty$ queue --- where the service discipline is modified.

{\bf $G/M/\infty$-queue with FIFO-batch departures.}
The following aspects are the same in the conventional $G/M/\infty$ model and the variant $G/M/\infty$ under study. There are infinitely many servers that provide service at a unit rate. An arriving customer immediately starts service at one of the non-occupied servers. Customers arrive according to a renewal process of rate $\lambda$, and their service requirements are exponentially distributed with parameter $\mu$, mutually independent and also independent from the arrival process. What is different in the two models is the departure mechanism. In the conventional $G/M/\infty$ queue, a customer leaves the system upon service completion, and hence customers depart one by one. In the model we analyze, departures occur in batches. We assume additionally that customers line up in the queue in the order of arrival. (This assumption is not restrictive but convenient when referring to departure events.) Suppose the service of a particular customer has finished. Then it is not this customer only who leaves the system but those {\it in front of}\, him in the queue leave as well. %The analogy with the blockchain update mechanism follows if communication of blocks in the Bitcoin network is viewed as service. Upon receiving a block that does not directly link to its blockchain version, the Bitcoin user requests the missing predecessor blocks and, similarly to a FIFO-batch departure, adds the entire batch of newly acquired blocks to his blockchain at once.

We refer to a customer together with those in front of him in the queue as {\it a~FIFO (first-in-first-out) batch}, and to the modified $G/M/\infty$ model described above --- as the {\it $G/M/\infty$ queue with FIFO-batch departures}.

%\begin{remark}
%The analogy between FIFO-batch departures and the blockchain update mechanism is the following. Upon receiving a block that does not directly link to his blockchain version, the Bitcoin user requests the predecessor blocks 
%\end{remark}

{\bf Alternative interpretation: $G/M/\infty$ queue with LIFO-batch departures.} 
Assume now that, instead of FIFO batches, {\it LIFO (last-in-first-out) batches} are formed upon departures. That is, if there is a service completion, the freshly served customer leaves together with those {\it behind}\, him in the queue.

\begin{comment}
Note that the two departure policies are equivalent in the sense that a FIFO-batch departure ``cuts off" a front part of the queue, and a LIFO-batch departure --- the tail of the queue, but in both scenarios the post-departure queue size is a uniform sample of the queue size prior to the departure. Because of this symmetry, the system dynamics is the same distributionally under either discipline.
\end{comment}

Note that the two departure policies are symmetric. A FIFO-batch departure ``cuts off" a front part of the queue, and a LIFO-batch departure --- the tail of the queue. But, in both scenarios, the post-departure queue size is a uniform sample of the queue size prior to the departure. This equivalence, in combination with the memorylessness of the exponential service times, implies that the queue dynamics is distributionally the same under either discipline.

More formally, between two consecutive arrivals, both the $G/M/\infty$ queue with FIFO-batch departures and that with LIFO-batch departures evolve like a continuous-time Markov chain on $\mathbb{Z}_+$ with transition rates
\begin{equation} \label{eq:rates_down}
q(i,j) = \mu, \quad i \geq 1, \ 0 \leq j \leq i-1,
\end{equation}
and absorption at the zero state. Respectively, overall, the two models evolve like the Markov chain~\eqref{eq:rates_down} that gets interrupted according to a renewal process that is independent from the chain, and every interruption (i.e., arrival) increases the state of the chain by one. In particular, this coupling implies that the two queue-length processes viewed as random elements of the Skorokhod space $\mathbf{D}$ have the same distribution.

\begin{remark} {\em As discussed in the introduction, FIFO-batch departures align with the initial motivation for the model that comes from the blockchain update mechanism. LIFO-batch departures, on the other hand, are a more convenient assumption in some of the proofs (in the rest of the proofs, a particular interpretation does not matter). Under LIFO-batch departures, a customer is affected by those in front of him in the queue only but not by those who arrive after him, which turns out to simplify the analysis. Although some of the results, for example the insensitivity of the busy period distribution, are not intuitively clear under FIFO-batch departures, all results of the paper are valid for both FIFO- and LIFO-batch departures by the coupling discussed above.%While, as discussed in the introduction, FIFO-batch departures imitate the blockchain update process and constitute the original motivation, LIFO-batch departures are a more convenient assumption in some of the proofs (in the rest of the proofs, a particular interpretation does not matter). Under LIFO-batch departures, a customer is affected by those in front of him in the queue only but not by those who arrive after him, which turns out to simplify the analysis. Although some of the results, for example, the insensitivity of the busy period distribution, are not intuitively clear under FIFO-batch departures, all results of the paper are valid for both FIFO- and LIFO-batch departures by the coupling discussed above. 
}\end{remark}

\begin{remark} {\em Borst et al.~\cite{PS-ROS}, too, employ a probabilistic coupling to show the equivalence of two service disciplines, processor sharing and random order of service, in the $G/M/1$ queue. In~\cite{PS-ROS}, each customer of one queue is coupled with a customer in the other queue. In this paper, we do not couple the two models directly but through a third model. Note that a straightforward modification of the coupling used here provides an alternative to the coupling used in~\cite{PS-ROS}.}
\end{remark}

{\bf More notation.} We conclude the model description by introducing some model-related notation used frequently throughout the paper.

The queue size at time $t \geq 0$ is denoted by $Q(t)$. The ``stationary" distribution of the queue-length process is denoted by $\pi_k$, $k \in \mathbb{Z}_+$, where the term ``stationary" has the conventional meaning if the arrival process is Poisson, and a slightly broadened meaning if the arrival process is renewal; the details follow in Section~\ref{sec:stry_dn}.  By $Q_\infty$ we denote a random variable such that $\pr \{Q_\infty  = k\} = \pi_k$, $k \in \mathbb{Z}_+$. Note that the (``stationary") distribution of the queue-length process does not depend on the particular interpretation of FIFO- or LIFO-batch departures.

By $E(\mu)$ we denote the generic service time of a customer (distributed exponentially with parameter $\mu$). More generally, $E(k\mu)$ stands for a random variable that is distributed exponentially with parameter $k \mu$, $k \geq 1$.

The generic inter-arrival time is denoted by $A$, its i.i.d.\ copies --- by  $A_1, A_2, \ldots$, and $S_n: = \sum_{i=1}^n A_i$, $n \geq 0$. We also use the normalized version of the inter-arrival time $\wtil A:= A/ \ex A$. Respectively, $\wtil A_1, \wtil A_2, \ldots$ stand for i.i.d.\ copies of $\wtil A$, and $\wtil S_n := \sum_{i=1}^n \wtil A_i$.

\begin{comment}
\begin{remark}{\em 
If the arrival process is Poisson, then the equivalence of FIFO- and LIFO-batch departures is straightforward because under both disciplines the queue-length process is a  continuous-time Markov chain on $\mathbb{Z}_+$ with transition rates
\begin{align} \label{eq:ctmc_rates}
\begin{split}
q(i, i+1) &= \lambda, \quad i \geq 0, \\
q(i,j) &= \mu, \quad i \geq 1, \  0 \leq j \leq i-1.
\end{split}
\end{align}}
\end{remark}
\end{comment}

\subsection{Busy period}
A {\it busy period} is a period of time between an arrival into the empty system and the first instant since this arrival when the system becomes empty again. Intuitively, a larger arrival rate should result in a longer busy period. Remarkably, in the model under study, the busy period distribution does not depend on the inter-arrival distribution (and in particular, it does not depend on the arrival rate).

\begin{lemma} \label{lem:busy_period}
Both under FIFO- and LIFO-batch departures, the busy period of the $G/M/\infty$ queue is distributed exponentially with parameter $\mu$.
\end{lemma}

\begin{proof}
Under FIFO-batch departures, the result seems counter-intuitive but it does admit an analytical proof, which we provide here for completeness. Denote by $E_0(\mu)$ the service time of the customer that initiates the busy period, by $E_1(\mu), E_2(\mu),$ $\ldots$ the service times of the newly arriving customers, and by $A_1, A_2, \ldots$ their inter-arrival times. The queue becomes empty if the last customer in the queue finishes his service before the next customer arrives. Hence, the busy period lasts for
$
B = S_{N-1} + E_{N-1}(\mu),
$
where
$
N := \min \{ n \geq 1 \colon E_{n-1}(\mu) \leq A_n \}.
$
But the memoryless property of the exponential distribution implies that
$
N \dcopy \min \{ n \geq 1 \colon E_0(\mu) \leq S_n \},
$
and hence $B \dcopy E_0(\mu)$.

However, the coupling of FIFO- and LIFO-batch departures discussed in Section~\ref{sec:coupling} allows for a more elegant proof. It ensures that the busy period distribution is the same under either discipline. Under LIFO-batch departures, in contrast to the analytical proof above,  the lemma  follows directly from the way customers interact. Indeed, a~customer leaves as soon as his own service time expires or the service time of someone in front of him in the queue. Hence, the customer who initiates the busy period leaves exactly when his own service time expires. But if he leaves, the queue empties. That is, the busy period lasts for exactly the service time of the initial customer. \end{proof}

\begin{comment}
{\color{blue} present the above as a proof or just text?}

{\color{blue}
LEAVE OR REMOVE?
\begin{remark} {\em The representation with LIFO-batch departures implies also that, for any initial state $n>0$, the time it takes the Markov process \eqref{eq:ctmc_rates} to reach the empty state has an exponential distribution with mean $1/\mu$.}
\end{remark}

\begin{remark}{\em 
For a Markov process with downward transitions
\[
q(i,j) = \mu, \quad i \geq 1, \quad 0 \leq j \leq i-1,
\]
only and no upward transitions,  and for any initial state $n > 0$, the time until absorption in the empty state has exponential distribution with mean $1/\mu$. It follows by the representation with FIFO-batch departures: the system remains non-empty as long as the last of the initial customers is present, and that is for the duration of his service time.}
\end{remark}
}
\end{comment}

\subsection{Stationary queue-length distribution}  \label{sec:stry_dn}
In this paper, by the {\it stationary queue-length distribution} we mean the long-run fractions of time the process $Q(\cdot)$ spends in particular states, i.e.\ the limits
\begin{equation} \label{eq:def_stat_dn}
\pr\{Q_\infty = k \} = \pi_k := \text{a.s.-}\lim_{t \to \infty} \dfrac{\int_0^t \ind \{ Q(s) = k \} ds}{t}, \quad k \in \mathbb{Z}_+.
\end{equation}

If the arrival process is Poisson, then $Q(\cdot)$ is a continuous-time Markov chain and the distribution~\eqref{eq:def_stat_dn} is stationary in the conventional sense: if $Q(0) \dcopy Q_\infty$, then $Q(t) \dcopy Q_\infty$ for all $t > 0$.

\begin{lemma} \label{lem:st_dn}  Both under FIFO- and LIFO-batch departures, the limits~\eqref{eq:def_stat_dn} exist and are given by 
\begin{equation} \label{eq:fla_stat_dn}
\sum_{i \geq k} \pi_i =  \pr \{Q_\infty \geq k \} =  \dfrac{1 - \ex e^{-k \wtil A/\rho}}{ k / \rho}\prod_{i=1}^{k-1} \ex e^{-i \wtil A / \rho}, \quad k \geq 1.
\end{equation}
In addition, 
\begin{equation} \label{eq:stat_dn_is_lim}
Q(t) \Rightarrow Q_\infty \quad \text{as $t \to \infty$}
\end{equation}
if and only if the inter-arrival time distribution is non-lattice {\em (}i.e., not concentrated on a set of the form $\{\delta, 2 \delta, \ldots \}${\em )}.%the ``stationary" distribution is also the limit distribution of the queue-length process,

\end{lemma}

Here we present the core ideas of the proof, and some further details follow in Section~\ref{sec:proof_stationary_analysis}. First, note that the queue-length process $Q(\cdot)$ is regenerative with the initiation instants of busy periods as regeneration points. Respectively, a cycle is a period of time between the initiation instants of two successive busy periods. In particular, if the inter-arrival time distribution is non-lattice, then the cycle distribution is non-lattice, too,  and then the limit queue-length distribution exists (see e.g.~\cite[Theorem VI.1.2]{Asmussen}) and~\eqref{eq:stat_dn_is_lim} necessarily holds.

Now we discuss why the limits~\eqref{eq:def_stat_dn} exist and satisfy the formula~\eqref{eq:fla_stat_dn}. The idea behind~\eqref{eq:fla_stat_dn} for $k=1$ is that the queue is non-empty when it is in busy periods, which are distributed exponentially with parameter $\mu$ and are parts of i.i.d.\ cycles. Denote by $C$ the generic cycle length. Then the law of large numbers implies that
\[
\sum_{i \geq 1} \pi_i = \text{a.s.-}\lim_{t \to \infty} \dfrac{\int_0^t \ind \{ Q(s) \geq 1 \} ds}{t} = \frac{\ex E(\mu)}{\ex C}.
\]

To derive~\eqref{eq:fla_stat_dn} for $k \geq 2$, we assume LIFO-batch departures and classify customers into levels: a customer is of level $k \geq 1$ if he arrives into the system with $k-1$ other customers. The point of  this classification is that the queue size is at least $k$ during the sojourn times of level $k$ customers. Note that the sojourn time (from arrival until departure) of a level $k$ customer is distributed exponentially with parameter $k\mu$ since it is the minimum of his own service time and the service times of the $k-1$ customers in front of him.  Also note that the sojourn times of different level $k$ customers do not overlap and are i.i.d., and that each level $k$ sojourn time is within a level $k-1$ sojourn time. Denote by $N_k$ the number of level $k$ customers that arrive during the sojourn time of a level $k-1$ customer. Then, by Wald's identity and the law of large numbers,
\[
\text{a.s.-}\lim_{t \to \infty} \frac{\int_0^t \ind \{ Q(s) \geq k \} ds}{\int_0^t \ind \{ Q(s) \geq k-1 \} ds} = \frac{\ex N_k \ex E(k\mu)}{\ex E((k-1)\mu)}, \quad k \geq 2,
\]
and hence,
\begin{equation} \label{eq:prelimiary_fla_stry_dn}
\text{a.s.-}\lim_{t \to \infty} \frac{\int_0^t \ind \{ Q(s) \geq k \} ds}{t} =  \frac{\ex E(k\mu)}{\ex C} \prod_{i=2}^k \ex N_i, \quad k \geq 1.
\end{equation}
It is left to compute the $\ex C$ and $\ex N_i$'s to get~\eqref{eq:fla_stat_dn}, which is postponed until Section~\ref{sec:proof_stationary_analysis}.

\subsection{Moments of the stationary queue length} With the additional assumption of Poisson arrivals, it is possible to derive a recursive relation for the factorial moments of the stationary queue-length, which are defined by
\[
m_0 := 1, \quad m_n := \mathbb{E} \prod_{i=0}^{n-1} (Q_\infty-i), \quad n \geq 1.
\]

\begin{lemma} \label{lem:st_dn_moments}
Under Poisson arrivals,
\begin{equation} \label{eq:moments_recursion}
m_{n+2} = (n+2)(\rho m_n - m_{n+1}), \quad n \geq 0.
\end{equation}
\end{lemma}

The base for the above recursion --- the mean $m_1 = \ex Q_\infty$ --- admits no closed-form expression but can be evaluated in a certain limit regime, see Section~\ref{sec:stry_asymptotics}.

As for the proof of~\eqref{eq:moments_recursion},  it follows as one incorporates the balance equations for the Markov (due to Poisson arrivals) process $Q(\cdot)$ into the generating function of its stationary distribution $\pi_k$, $k \in \zplus$, and applies a Taylor expansion around 1. The detailed derivation follows in Section~\ref{sec:proof_stationary_analysis}.

\subsection{Asymptotics of the stationary queue length} \label{sec:stry_asymptotics} The final result  of this section is that, for large $\rho$, i.e.\ in presence of service delays, the stationary queue length is of order $\sqrt{\rho}$. This asymptotic behavior motivates the scaling limit of the queue length process that we develop in the next section.

\begin{lemma} \label{th:st_dn_FL}
As $\rho \to \infty$,
$
Q_\infty/ \sqrt{\rho} \Rightarrow \zeta,
$
where the limit random variable $\zeta$ has the following Weibull {\em (}Rayleigh{\em )} distribution:
\[
\pr \{ \zeta \geq x \} = e^{-\frac{x^2}{2}}, \quad x > 0.
\]
Moreover,
$
\ex Q_\infty / \sqrt{\rho} \to \ex \zeta  = \sqrt{\pi/2}.
$
\end{lemma}

The proof of this lemma is provided in Section~\ref{sec:proof_stationary_analysis}. We asymptotically characterize the tail  $\pr \{  Q_\infty / \sqrt{\rho} \geq x\}$ and the mean $\ex Q_\infty = \sum_{k \geq 1} \pr \{ Q_\infty \geq k \}$ based on the formula~\eqref{eq:fla_stat_dn} and the dominated convergence theorem.

\begin{comment}
Note that the asymptotics for the mean implies the asymptotics for the higher moments by Lemma~\ref{lem:st_dn_moments}.
\end{comment}

\begin{corollary}
Under Poisson arrivals,
$
\ex \left( Q_\infty/ \sqrt{\rho} \right)^n \to \ex \zeta^n  = n!! (\sqrt{\pi/2})^{n \bmod 2}
$, $n \geq 1$,
where the double factorial is defined by $n!! := \prod_{\begin{subarray}{l}
1 \leq l \leq n, \\
 l \equiv n \, (\textup{mod} \, 2)
 \end{subarray}}  l$.
\end{corollary}

Lemma~\ref{th:st_dn_FL} and the recursion~\eqref{eq:moments_recursion} imply the last corollary of this section.

\section{Transient analysis} \label{sec:transient_analysis}
In this section, we characterize the evolution of the $G/M/\infty$ queue with FIFO/LIFO-batch departures in time. 

\subsection{Direct transient analysis} Under the additional assumption of Poisson arrivals, it is possible to find the queue-length distribution at a fixed time instant by solving the Kolmogorov equations.
\begin{lemma} \label{lem:tr_dn}
Under Poisson arrivals,
\begin{equation} \label{eq:kolm_sol}
\pr \{ Q(t)  = k \} = C_{k,0} + \sum_{i=1}^{k+1} C_{k,i} e^{-(\lmb+i\mu)t}, \quad t > 0,
\end{equation}
where the coefficients $C_{k,i}$ can be computed recursively as follows: for $k \geq 0$,
\begin{subequations} \label{eq:kolm_induction}
\begin{align}
C_{k,0} &=  \pr\{ Q_\infty = k \} = \pi_k, \label{eq:kolm_const} \\
(k+1-i)C_{k,i} &= \rho C_{k-1,i} - \sum_{l=i-1}^{k-1} C_{l,i}, \quad 1 \leq i \leq k, \label{eq:kolm:middle} \\
C_{k, k+1} &= \pr\{ Q(0) = k \} - \sum_{i=0}^k C_{k,i} \label{eq:kolm_last}.
\end{align}
\end{subequations}
\end{lemma}

The derivation of these formulae can be found in the appendix. Note that~\eqref{eq:kolm_const}  follows immediately since the stationary queue-length probabilities $\pi_k$ are also the limit queue-length probabilities as $t \to \infty$. As an aside, we mention that the above recursion allows the computation of $\pr\{ Q(t) \leq k \}$ in $O(k^3)$ operations.

\subsection{Fluid limit} While the last lemma gives the queue-length distribution at a fixed time instant, 
the next result characterizes the entire trajectory of the queue-length process $Q(\cdot)$. We treat $Q(\cdot)$ as a random element of the Skorokhod space $\mathbf{D}$ and develop a distributional approximation for it as $\rho \to \infty$.  Lemma~\ref{th:st_dn_FL} suggests that space should be scaled with $\sqrt{\rho}$ as $\rho \to \infty$, and the following observation motivates the time scaling. Suppose the queue starts empty. Then the number $Y_1$ of customers that accumulate in the queue before the first departure has the distribution
\[
\pr \{ Y_1 \geq k \} = \prod_{i=0}^{k-1} \pr \{ A < E(i \mu) \} = \prod_{i=1}^{k-1} \ex^{-i \wtil A / \rho}, \quad k \geq 1.
\]
In particular, by~\eqref{eq:fla_stat_dn} and Lemma~\ref{th:st_dn_FL}, as $\rho \to \infty$,
\[
\pr \left\{  \frac{Y_1}{\sqrt{\rho}} \geq x \right\} \sim \pr \left\{  \frac{Q_\infty}{\sqrt{\rho}} \geq x \right\} \to e^{-\frac{x^2}{2}}, \quad x > 0.
\]
That is, an empty queue builds up to a level of order $\sqrt{\rho}$ until the first departure. If time is renormalized so that that the arrival rate $\lambda$ is fixed and the service rate $\mu \to 0$, then the time of the first departure is of order $\sqrt{\rho}$ as well. Hence, we consider the family of the scaled processes
\begin{equation} \label{eq:scaling}
\lambda = \textup{const}, \quad \mu \to 0, \quad \overline{Q}^\rho(t) := \frac{Q(\sqrt{\rho}t)}{\sqrt{\rho}}, \quad t \geq 0.
\end{equation}
The above scaling is a law-of-large-numbers scaling (or: fluid scaling), since space and time are scaled with the same large factor. Limit processes that arise under such scalings are usually referred to as {\it fluid limits}. Typically, Markov processes have deterministic fluid limits, similarly to the deterministic limit in the law of large numbers. For the model under study, however, the fluid limit is {\it random}, as is stated in the next theorem.

\begin{theorem} \label{th:process_FL}
Suppose that $\overline{Q}^\rho(0) \Rightarrow \xi_0$ as $\rho \to \infty$, where the limit initial condition $\xi_0$ is either deterministic or absolutely continuous. Then the scaled processes $\overline{Q}^\rho(\cdot)$ converge weakly in the Skorokhod space $\mathbf{D}$ to a %stochastic process
random fluid limit $\overline{Q}(\cdot)$ that exhibits downward jumps of random size at random instants and grows linearly at rate $\lmb$ between the jumps.

For $n \geq 1$, denote by $\tau_n$ the instant of the $n$-th jump of $\overline{Q}(\cdot)$ and put $\tau_0:=0$. The  post- and pre-jump levels $\overline{Q}(\tau_{n-1})$ and $\overline{Q}(\tau_n-)$ form a Markov chain that starts from $\overline{Q}(\tau_0) = \overline{Q}(0) = \xi_0$ and evolves according to the following transition probabilities: for $n \geq 1$,
\begin{subequations} \label{eq:FL_MC}
\begin{align}
& \pr \{ \overline{Q}(\tau_n-) \geq y | \overline{Q}(\tau_{n-1}) = x\} = e^{\frac{-y^2+x^2}{2}}, \quad y > x,  \label{eq:12}\\
&\pr \{ \overline{Q}(\tau_n) \leq x | \overline{Q}(\tau_n-) = y\} = \frac{x}{y}, \quad 0 < x < y \label{eq:13}.
\end{align}
\end{subequations}
The above Markov  chain determines the jump instants and the entire trajectory of the fluid limit: for $n \geq 1$, 
\begin{align*}
&\tau_n - \tau_{n-1} =  \frac{1}{\lmb}\big(\overline{Q}(\tau_n-)-\overline{Q}(\tau_{n-1})\big), \\
&\overline{Q}(t) = \overline{Q}(\tau_{n-1}) + \lmb (t-\tau_n), \quad t \in [\tau_{n-1}, \tau_n).
\end{align*}

In addition, 
\begin{equation} \label{eq:25}
\tau_n \to \infty \quad \text{a.s.\ as $n \to \infty$}.
\end{equation}
\end{theorem}

\begin{remark} \label{rem:nonexplosive}{\em 
The fluid limit $\overline{Q}(\cdot)$ in Theorem~\ref{th:process_FL} belongs to the class of growth-collapse/ stress release processes. In the related literature, the property~\eqref{eq:25} is referred to as {\it non-explosiveness}. In particular, the process $\overline{Q}(\cdot)$ is non-explosive by Last~\cite[Theorem 3.1]{Last}. }
\end{remark}

\begin{figure}[!b] 
\centering
\includegraphics[scale=0.75]{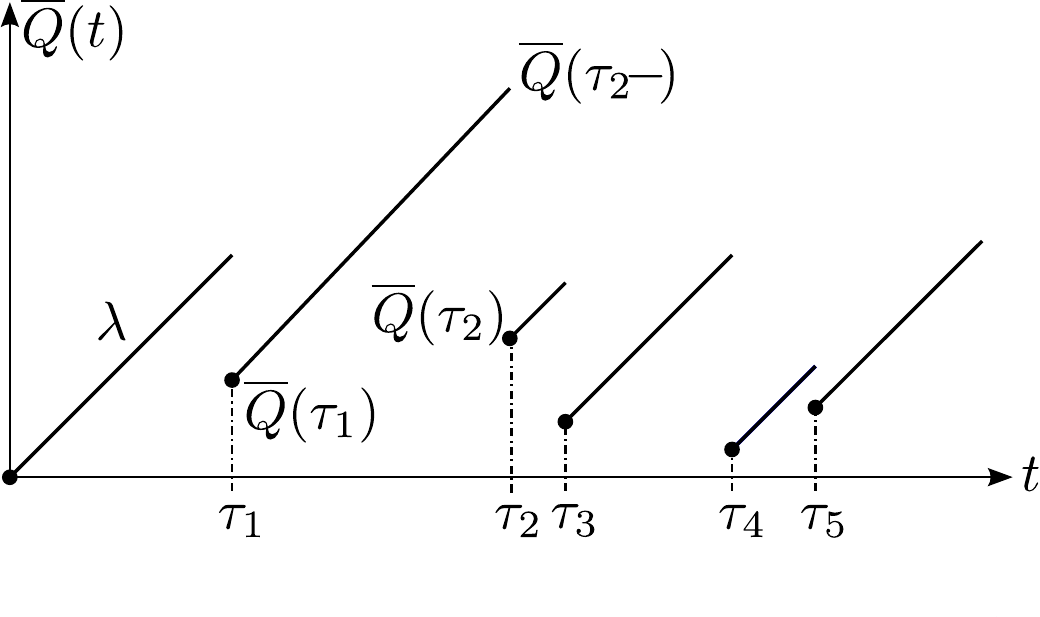}
\caption{A trajectory of the fluid limit.}
\label{fig:fl}
\end{figure}

The structure of the fluid limit aligns with the evolution of the pre-limit model as follows (see Figure~\ref{fig:fl} below and also Figure~\ref{fig:post_pre_departure_levels} in Section~\ref{sec:proof_FL}): the jump instants of the fluid limit and its post-/pre-jump levels are analogues of the departure instants and the post-/pre-departure queue levels. In fact, the former can be shown to arise as the scaling limits of the latter. Although this result is natural, it is tricky to prove directly. The renewal arrivals make the queue-length process non-Markov. In particular, the post- and pre-departure queue levels are not straightforward to analyze. Instead (see Section~\ref{sec:proof_FL} for the proof of Theorem~\ref{th:process_FL}), we  consider the queue levels in the beginning and at the end of those inter-arrival periods where departures occur. These queue levels do form a Markov chain that we show to converge weakly under the fluid scaling to the Markov chain formed by the post- and pre-jump fluid limit levels. Then we additionally prove that, with high probability, there is at most a single departure instant per inter-arrival period. That is, indeed, the post- and pre-jump fluid limit levels correspond one-to-one to the post- and pre-departure queue levels. 

The mentioned convergence of the embedded Markov chain  lays the foundation for the proof of the weak convergence $\overline{Q}^\rho(\cdot) \Rightarrow \overline{Q}(\cdot)$ in the Skorokhod space $\mathbf{D}$. We establish two  ingredients that imply weak convergence in $\mathbf{D}$:  convergence of finite-dimensional distributions and relative compactness of the pre-limit processes. To show the relative compactness,  we follow a conventional approach (see e.g.~Ethier and Kurtz \cite{EthierKurtz}) and verify that the pre-limit processes $\overline{Q}^\rho(\cdot)$ remain bounded on finite intervals and have small oscillations on small intervals.

\begin{remark} {\em 
A FIFO-/LIFO-batch departure is effectively a uniform cut-off, and this property is inherited by the fluid-limit: we have $\overline{Q}(\tau_n) = U_n \overline{Q}(\tau_n-)$, where $U_n$, $n \geq 1$, are i.i.d.\ random variables distributed uniformly on the interval $(0,1)$. The proof of Theorem~\ref{th:process_FL} in Section~\ref{sec:proof_FL} generalizes in a straightforward way to accommodate the following service discipline: whenever there is a service  completion, a number of customers still in service interrupt their service and leave together with the freshly served customer so that a generally distributed fraction $B$ of customers  remains in the system. The fluid limit of the $G/M/\infty$ under this new discipline  then satisfies $\overline{Q}(\tau_n) = B_n \overline{Q}(\tau_n-)$, where $B_n$, $n \geq 1$, are i.i.d.\ copies of $B$. }\end{remark}

\subsection{Stationary fluid limit} \label{sec:stationary_FL}
\begin{comment}

Usually, the fluid limit of a Markov process solves a deterministic differential equation that imitates the original stochastic dynamics. In addition, the stationary distribution of a Markov process usually converges under fluid scaling to the fixed point of the fluid limit differential equation, provided it is unique.

The queueing model under study does not fall into the above category, and in this section, we find out how its fluid limit $\overline{Q}(\cdot)$ and the limit stationary distribution $\zeta$ are related. Due  its structure, the fluid limit $\overline{Q}(\cdot)$ has no fixed point. Note that we do not check whether the fluid limit  $\zeta$ is a fixed point for the fluid limit in the weak sense, i.e.\ whether $\overline{Q}(0) \overset{\text{d}}{=} \zeta$ implies $\overline{Q}(t) \overset{\text{d}}{=} \zeta$ for all $t \geq 0$.

Instead, we adapt the notion of stationarity to the specific structure of the fluid limit.
\end{comment}

As mentioned before, the random fluid limit $\overline{Q}(\cdot)$  in Theorem~\ref{th:process_FL} belongs to the class of growth-collapse processes, which have been studied in the literature. In particular, there are results that characterize stationary and long-time limit behavior of such processes. Note that, in the context of growth-collapse processes like $\overline{Q}(\cdot)$, it is natural to consider the stationary distribution of the process itself and of the Markov chain of the post- and pre-jump levels, and to identify relations between the two stationary distributions. In the next lemma, we summarize these types of results for the fluid limit $\overline{Q}(\cdot)$. In addition to references, we also provide a simple alternative to the more general but also more involved derivation of the stationary post- and pre-jump distribution available in the literature.

 \begin{lemma} \label{lem:stationary_FL}
{\em (a)} The unique stationary distribution of the fluid limit $\overline{Q}(\cdot)$, which is also its limit distribution as $t \to \infty$, is that  of the limit random variable $\zeta$ from Lemma~$\ref{th:st_dn_FL}$.

{\em (b)}  The unique stationary distribution of the Markov chain $(\overline{Q}(\tau_{n-1}),\overline{Q}({\tau_n-}))$, $n \geq 0$, which is also its limit distribution as $n \to \infty$, is absolutely continuous. Let a pair of random variables $(\xi,\eta)$ have this stationary/limit distribution, then their marginal densities are given by
\[ f_\xi(x) = \sqrt{\frac{2}{\pi}} e^{-\frac{x^2}{2}},  \quad f_\eta(y) = \sqrt{\frac{2}{\pi}} y^2 e^{-\frac{y^2}{2}}, \quad x,y> 0. \\
\]
{\em (c)}  The distributions of $\zeta$ and $(\xi$, $\eta)$ are related as follows: for any measurable non-negative function $f(\cdot)$,
\[
\ex f(\zeta) = \frac{\ex \int_\xi^\eta f(u) du}{\ex(\eta-\xi)}.
\]
\end{lemma}

\begin{remark}{\em 
By Lemmas~\ref{lem:st_dn}, \ref{th:st_dn_FL} and~\ref{lem:stationary_FL}(a), for the $G/M/\infty$ queue with FIFO-/LIFO-batch departures, the limits as $\rho \to \infty$ and $t \to \infty$ commute if the inter-arrival distribution is non-lattice.}
\end{remark}

\begin{comment}
\begin{lemma}
The limit stationary distribution $\zeta$ is the unique stationary and limit distribution for the Markov process $\overline{Q}(\cdot)$.
\end{lemma}
\end{comment}

Note that, as Theorem~\ref{th:process_FL} is applied to the initial conditions $\overline{Q}^\rho(0) \dcopy Q_\infty / \rho \Rightarrow \zeta$, $\rho \to \infty$, it follows that the distribution of $\zeta$ is stationary for the fluid limit $\overline{Q}(\cdot)$. The uniqueness of this stationary distribution and convergence to it as $t \to \infty$ (and hence (a)) follow by  Boxma et al.~\cite{OnnoGC} who prove these types of results for a class of growth-collapse processes that includes $\overline{Q}(\cdot)$. Their proof views a growth-collapse process as regenerative with regeneration instants being returns to a particular level, and shows that the mean return time to any level is finite. In combination with (b), this implies (c)  since
\[
\ex f(\zeta) = \text{a.s.-}\lim_{t \to \infty} \frac{\int_0^t f(\overline{Q}(s))ds}{t} = \frac{\ex \int_\xi^\eta f(u) du}{\ex(\eta-\xi)},
\]
where the first identity is by the ergodic theorem for regenerative processes, and the second --- by the ergodic theorem for the Markov chain $(\overline{Q}(\tau_{n-1}), \overline{Q}(\tau_n-))$, $n \geq 1$.

The stationary densities in (b) can be obtained from a more general result by Guillemin~\cite[Section 3, Example 2]{AIMD2}. We include here a more straighforward derivation.

\begin{proof}[Proof of (b)] The relation \eqref{eq:12} means that the jump rate of the fluid limit $\overline{Q}(\cdot)$ is $\lambda x$ when it is in state $x$. Hence, the post-jump levels satisfy \[\left(\overline{Q}(\tau_n)\right)^2 = U^2_n\left(\left(\overline{Q}(\tau_{n-1})\right)^2+2E_n(1)\right), \quad n \geq 1, \] where the $U_n$'s are distributed uniformly on the interval $(0,1)$, the $E_n(1)$'s are distributed exponentially with parameter 1, all mutually independent (in detail, this derivation is conducted, for example, by Chafa\"\i~et al.~\cite{Malrieu}). Then the unique stationary and limit distribution for the Markov chain $\left(\overline{Q}(\tau_n)\right)^2$, $n \geq 0$,  is that of $\xi^2 \dcopy \sum_{n \geq 1} E_n(1) \prod_{i=1}^n U_i^2$. In particular, this representation implies that $\xi$ is absolutely continuous and has a finite mean. Then we can find the density $f_\xi(\cdot)$ by simply solving a differential equation. 

Indeed, by the relations~\eqref{eq:FL_MC}, the stationary post-jump density $f_\xi(\cdot)$ solves the integral equation
\[
f_\xi(x) = \frac{1}{\rho} \int_x^\infty e^{-\frac{y^2}{2\rho}} \left( \int_0^y e^{\frac{u^2}{2\rho}} f_\xi(u) du \right) dy, \quad x > 0,
\]
which is equivalent to the differential equation
\begin{equation} \label{eq:17}
f''_\xi(x) + x f'_\xi(x)+f_\xi(x) = 0, \quad x > 0.
\end{equation}
The general solution to~\eqref{eq:17} is
$
f_\xi(x) = C_1 e^{-x^2/2} + C_2 D (x/\sqrt{2}),
$
where $C_1$, $C_2$ are constants and  $D(\cdot)$ is the Dawson function given by
$
D(x) = e^{-x^2} \int_0^x e^{u^2} du.
$
Note that $D(x) \sim 1/(2x)$ as $x \to \infty$, and hence $\xi$ has a finite mean if and only if $C_2=0$, i.e.
$
f_\xi(x) = \sqrt{2/\pi} \,e^{-x^2/2}.
$
\end{proof}

\section{Proofs for Section~\ref{sec:stationary_analysis}} \label{sec:proof_stationary_analysis}
In this section, we fill in the details of the proof of Lemma~\ref{lem:st_dn} (the main ideas of that proof were already given in Section~\ref{sec:stry_dn}), and also we prove Lemma~\ref{th:st_dn_FL} here.

\begin{comment}

\begin{proof}[Proof of Lemma \ref{lem:busy_period} under FIFO-batch departures] Denote by $E_0(\mu)$ the service time of the customer that initiates the busy period, by $E_1(\mu), E_2(\mu), \ldots$ the service times of the newly arriving customers, and by $A_1, A_2, \ldots$ their inter-arrival times. The queue becomes empty if the last customer in the queue finishes his service before the next customer arrives. Hence, the busy period lasts for
$
B = S_{N-1} + E_{N-1}(\mu),
$
where
$
N := \min \{ n \geq 1 \colon E_{n-1}(\mu) \leq A_n \}.
$
But the memoryless property of the exponential distribution implies that
$
N \dcopy \min \{ n \geq 1 \colon E_0(\mu) \leq S_n \},
$
and hence $B \dcopy E_0(\mu)$.
\end{proof}

\end{comment}

\begin{proof}[Proof of Lemma~\ref{lem:st_dn}]
In Section~\ref{sec:stry_dn}, we have discussed why the non-latticeness of the inter-arrival distribution is sufficient for the limit queue-length distribution to exist, and here we show that it is also necessary. Indeed, if the inter-arrival distribution is concentrated on a grid $\{\delta, 2 \delta, \ldots\}$ (e.g.\ if the inter-arrival time is deterministic),  then, for any $\eps \in (0,\delta)$,
\[
\pr \{ Q(n \delta +\eps) = 0 \} = \pr \{ Q(n\delta) = 0 \} + (1-\pr \{ Q(n \delta) = 0 \}) \pr \{ E(\mu) < \eps \},
\]
and hence the limits $\lim_{n \to \infty} \pr \{ Q(n\delta+\eps) = 0 \}$ and $\lim_{n \to \infty} \pr \{ Q(n\delta) = 0 \}$ can not be equal.
\begin{comment}The  only possibility for the two limits to be the same is to be 1. This contradicts the fact that
\begin{align*}
\pr \{ Q(\delta n)  = 0 \} \leq \pr \{  \text{no arrival at $\delta n$}\} = 1 - \pr\{ \text{arrival at $\delta n$} \},
\end{align*}
where $\pr\{ \text{arrival at $\delta n$}\}$ has a non-zero or no limit as $n \to \infty$ (see e.g.~\cite[Theorem I.2.2 and Corollary I.2.3]{Asmussen}).
\end{comment}

The remainder of the proof  establishes \eqref{eq:fla_stat_dn}. In Section~\ref{sec:stry_dn}, we derive the preliminary formula~\eqref{eq:prelimiary_fla_stry_dn}, and here we compute the $\ex C$ and  $\ex N_k$'s, see the formulae \eqref{eq:mean_C} and~\eqref{eq:mean_N_k} below. As one plugs \eqref{eq:mean_C} and \eqref{eq:mean_N_k}  into~\eqref{eq:prelimiary_fla_stry_dn}, the expression \eqref{eq:fla_stat_dn} follows.

First we find the mean cycle length $\ex C$. Under LIFO-batch departures, a busy period lasts as long as the service time of the customer that has initiated the busy period lasts. The first arrival following the service completion of this customer initiates the next busy period. Hence,
$
C \dcopy S_N,
$
where
$
N: = \min\{ n \geq 1 \colon S_n \geq E(\mu) \}
$
and $E(\mu)$  is independent from $\{S_1, S_2, \ldots \}$. (This representation implies, in particular, that the cycle distribution is non-lattice if and only if the inter-arrival distribution is non-lattice.) We have
\[
\pr \{ N = n \} = \pr \{ A < E(\mu) \}^{n-1} \pr \{ A \geq E(\mu) \} = (\ex e^{-\mu A})^{n-1} (1-\ex e^{-\mu A}),
\]
and then Wald's identity gives
\begin{equation} \label{eq:mean_C}
\ex C = \ex N \,\ex A = \frac{\ex A}{1 - \ex e^{-\mu A}} = \frac{1}{\lmb (1 - \ex^{-\wtil A / \rho})}.
\end{equation}
Now we compute the mean number $\ex N_k$ of level $k$ arrivals during a level $k-1$ sojourn time, $k \geq 2$. Define a level $k$ cycle to be a period of time between two successive level $k$ arrivals during a level $k-1$ sojourn time, and denote by $C_k$ the generic level $k$ cycle length. Then
\[
\pr \{ N_k \geq n \} = \pr \{ A < E((k-1)\mu) \} \pr\{ C_k < E((k-1)\mu) \}^{n-1}, \quad n \geq 1,
\]
where $X$ is independent from $E((k-1)\mu)$ and $C_k$ is not. By specifying the number $i$ of arrivals during a level $k$ cycle, we get
\begin{align*}
\pr \{ C_k < E((k-1)\mu) \}
&=\sum_{i \geq 1}   \pr \{ A < \min (E(\mu), E((k-1)\mu)) \}^{i-1} \pr \{ E(\mu) \leq A < E((k-1)\mu) \} \\
&= \frac{\pr \{ E(\mu) \leq A < E((k-1)\mu) \}}{1-\pr \{ A < \min(E(\mu), E((k-1)\mu) ) \}},
\end{align*}
where $E(\mu)$, $E((k-1)\mu)$ and $A$ are mutually independent, $E(\mu)$ represents the service time of the customer that has initiated a level $k$ cycle, $E((k-1)\mu)$ --- the level $k-1$ sojourn time that contains this level $k$ cycle, and $A$ --- an inter-arrival time during this level $k$ cycle. The above yields
\begin{align}
\ex N_k &= \sum_{n \geq 1} \pr \{ N_k \geq n \} = \frac{\pr \{ A< E((k-1)\mu) \}}{1 - \pr \{ C_k < E((k-1)\mu) \}}  \nonumber \\
&= \frac{\pr \{ A < E((k-1)\mu) \} (1 - \pr\{ A < E(k\mu)  \})}{1 - \pr \{A < E((k-1)\mu) \}} = \frac{\ex e^{-(k-1) \wtil A / \rho} (1 - \ex e^{-k \wtil A / \rho})}{1 - \ex e^{-(k-1) \wtil A / \rho}} \label{eq:mean_N_k}.
\end{align}
The proof of the lemma is now complete.
\end{proof}

The following three propositions are auxiliary results that are used in the proofs of Lemma~\ref{th:st_dn_FL}, and also of Theorem~\ref{th:process_FL} in Section~\ref{sec:transient_analysis}. We provide brief comments about the proofs of these three statements.

\begin{proposition} \label{prop:swap}
For any $n \geq k \geq 1$,
$
\sum_{i=k}^n i \wtil{A}_i \mathop{\rm =}\limits^{\textup{d}} (k-1) \wtil S_{n-k+1}+\sum_{i=1}^{n-k+1}  \wtil{S}_i.
$
\end{proposition}

The result follows as $\wtil A_i$ in the sum is replaced by $\wtil A_{n-i+1}$, $k \leq i \leq n$.

\begin{proposition} \label{prop:lln} {\em (a)}  As $n \to \infty$, $\sum_{i=1}^n \wtil S_i / n^2 \to 1/2$ a.s. {\em (b)}  Let $\alpha> 0$ and $n \colon \nneg \to \zplus$, $n(\rho) \sim \alpha \sqrt{\rho}$ as $\rho \to \infty$. Then $\ex e^{-\sum_{i=1}^{n(\rho)} \wtil S_i / \rho} \to e^{-\alpha^2/2}$.
\end{proposition}

\begin{proposition} \label{prop:exp_fn}
Let $\alpha > 0$. {\em (a)}  The function
\[
f(x) = \left\{\begin{array}{ll}
        \alpha, & \text{if } x = 0,\\
        \dfrac{1 - e^{-\alpha x}}{x}, & \text{if } x > 0,
        \end{array}\right.
\]
is decreasing and continuous. {\em (b)}  As $x \to 0$, $\ex f(\wtil A) \to 1$.
\end{proposition}

Both in Proposition~\ref{prop:lln} and Proposition~\ref{prop:exp_fn}, part (b) follows from part (a) by the dominated convergence theorem.

Now we are in a position to prove Lemma~\ref{th:st_dn_FL}.

\begin{proof}[Proof of Lemma \ref{th:st_dn_FL}] The weak convergence result  follows since, by Lemma~\ref{lem:st_dn} and Proposition~\ref{prop:swap},
\[
\pr \left\{ \dfrac{Q_\infty}{\sqrt{\rho}} \geq x \right \} = \ex \frac{1 - e^{-\wtil{A} \lc x \sqrt{\rho} \rc / \rho }}{\lc x \sqrt{\rho} \rc / \rho} \ex e^{ -\sum_{i=1}^{ \lc x \sqrt{\rho} \rc} \wtil{S_i} / \rho }, \quad x > 0,
\]
and then Propositions \ref{prop:lln}(b) and \ref{prop:exp_fn}(b) imply that
$
\pr \left\{ Q_\infty/\sqrt{\rho} \geq x \right \} \to e^{-x^2/2}
$ as $\rho \to \infty$.

To prove the second part of the lemma, which states convergence of the mean, we use the formula
\[
\ex Q_\infty =  \sum_{k \geq 1} \pr \{ Q_\infty \geq k \} = \sum_{k \geq 1} \ex \frac{1-e^{-\wtil A k / \rho}}{k / \rho} \ex e^{- \sum_{i=1}^{k-1} \wtil S_i / \rho}.
\]
Fix a $T>0$ and split the sum above into the sum over $k \leq T \sqrt{\rho}$ and over $k > T \sqrt{\rho}$. Then, by Proposition~\ref{prop:exp_fn}(a),
\begin{equation} \label{eq:18}
\ex \frac{1 - e^{-\wtil A T/ \sqrt{\rho}}}{T/\sqrt{\rho}} \ex L(\rho) \leq \frac{\ex Q_\infty}{\sqrt{\rho}} \leq \ex L(\rho) + U(\rho),
\end{equation}
where
\[
L(\rho) := \frac{1}{\sqrt{\rho}} \sum_{1 \leq k \leq T \sqrt{\rho}} e^{- \sum_{i=1}^{k-1} \wtil S_i / \rho}, \quad U(\rho) :=  \frac{1}{\sqrt{\rho}} \sum_{k > \sqrt{\rho} T} \ex e^{- \sum_{i=1}^{k-1} \wtil S_i / \rho}.
\]
Now we find the limits of $\ex L(\rho)$ and $U(\rho)$ as $\rho \to \infty$. Proposition~\ref{prop:lln}(a) implies that
$
L(\rho) \to \int_0^T e^{-x^2/2} dx
$
a.s.\ as $\rho \to \infty$,
and because of the uniform boundedness $L(\rho) \leq T$, $\rho > 0$,  we also have
\begin{equation} \label{eq:19}
\ex L(\rho) \to \int_0^T e^{-x^2/2} dx \quad \text{as $\rho \to \infty$}.
\end{equation}
To find the limit of $U(\rho)$, we split the sum in its definition into the sums over the sets of the form  $n \lf T \sqrt{\rho} \rf < k \leq (n+1) \lf T \sqrt{\rho} \rf$ and get
\begin{align*}
U \leq \frac{\lf T \sqrt{\rho} \rf}{\sqrt{\rho}} \sum_{n \geq 1} \ex e^{- \sum_{i=1}^{n \lf T \sqrt{\rho} \rf} \wtil S_i / \rho} \leq T \sum_{n \geq 1} \left ( \ex e^{- \sum_{i=1}^{\lf T \sqrt{\rho} \rf} \wtil S_i / \rho}  \right)^n
 \leq T \frac{\ex e^{- \sum_{i=1}^{\lf T \sqrt{\rho} \rf} \wtil S_i / \rho}}{1-\ex e^{- \sum_{i=1}^{\lf T \sqrt{\rho} \rf} \wtil S_i / \rho}} .
\end{align*}

The last inequality and \eqref{eq:18}--\eqref{eq:19}, in combination with Propositions~\ref{prop:lln}(b) and \ref{prop:exp_fn}(b), imply that, for any $T > 0$,
\[
\int_0^T e^{-x^2/2} dx \leq \liminf_{\rho \to \infty} \frac{\ex Q_\infty}{\sqrt{\rho}} \leq \limsup_{\rho \to \infty} \frac{\ex Q_\infty}{\sqrt{\rho}} \leq \int_0^T e^{-x^2/2} dx + T \frac{e^{-T^2/2}}{1 - e^{-T^2/2}},
\]
and then, by sending $T \to \infty$, we obtain
$
\ex Q_\infty/\sqrt{\rho} \to \int_0^\infty e^{-x^2/2} dx = \sqrt{\pi/2}
$.
\end{proof}

% PROOFS OF FLUID LIMIT THEOREM
\section{Proof of Theorem~\ref{th:process_FL}}  \label{sec:proof_FL}
 \begin{comment} Figure 1 below helps visualize the formal definitions that we are about to give. \end{comment}
 
In this section, we prove the fluid limit theorem of Section~\ref{sec:transient_analysis}. We first introduce some random quantities and events that are key for the analysis.% carried out here. 

Without loss of generality, assume throughout that no departure occurs simultaneously with an arrival, and hence any departure is within an inter-arrival period $(S_{n-1}, S_n)$. Define the indices (see also Figure~\ref{fig:post_pre_departure_levels})
\[
N_0:=0, \quad N_l := \min\{ n \geq N_{l-1}+1 \colon Q(S_n-) < Q(S_{n-1}) \}, \quad l \geq 1;
\]
that is, $(S_{N_l-1}, S_{N_l})$ is the $l$-th inter-arrival period that contains a departure.

\begin{figure}[!b]
\centering
\includegraphics[scale=0.75]{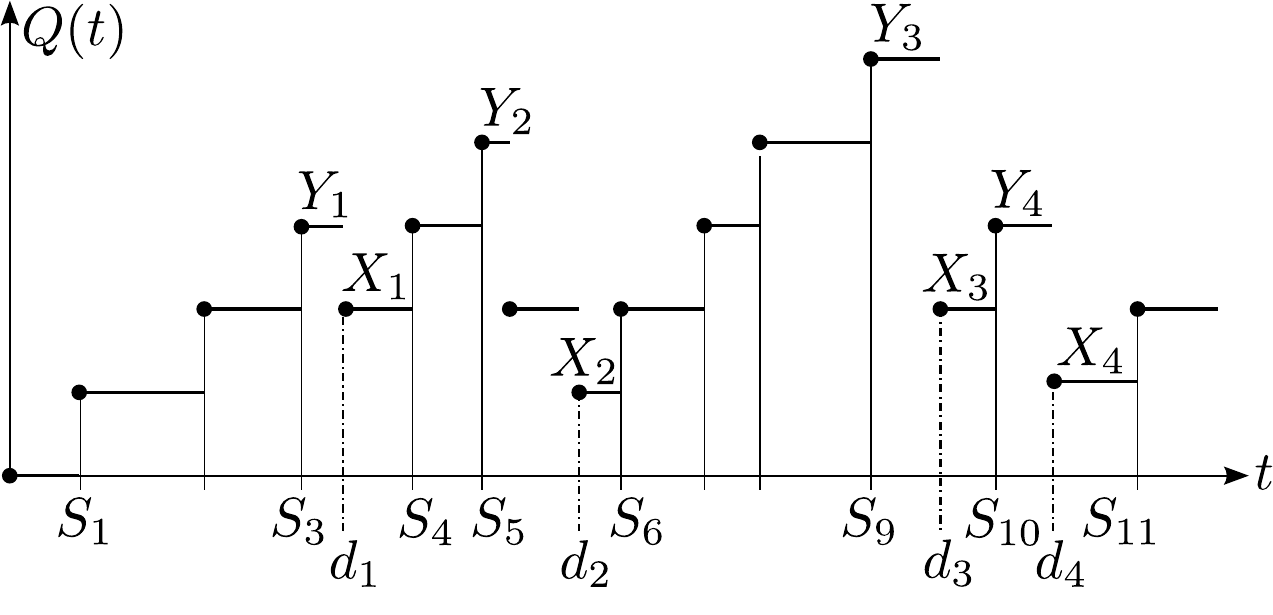}
\caption{For this realization, $N_1 = 4$, $N_2 = 6$, $N_3 = 10$, and $N_4 = 11$. The events $D_1$, $D_3$ and $D_4$ have occured in this realization, but the event  $D_2$ has not.}
\label{fig:post_pre_departure_levels}
\end{figure}

Next, put (again, see Figure~\ref{fig:post_pre_departure_levels})
\[
X_0 := Q(0), \quad Y_l := Q(S_{N_l-1}), \quad X_l := Q(S_{N_l}-),  \quad l \geq 1;
\]
that is, $Y_l$ is the queue size in the beginning of the $l$-th inter-arrival period that contains a departure, and $X_l$ is the queue size in the end of this period.
Somewhat loosely, we refer to the $Y_l$'s and $X_l$'s as {\it pre- and post-departure queue levels}, respectively. They are prototypes of the pre- and post-jump fluid limit levels $\overline{Q}(\tau_l-)$ and $\overline{Q}(\tau_l)$. Both under FIFO- and LIFO-batch departures, the sequence $X_0$, $(Y_n,X_n)$, $n \geq 1$, forms a Markov chain. To describe the evolution of this Markov chain, define the events
%\begin{align*}
%E_i := \{ \exists j =1, \ldots, i \colon E_j(\mu) < A \}, \quad i \geq 1,
%\end{align*}
\begin{align*}
E_i := \{ E_j(\mu) < A \text{ for some $j = 1, \ldots, i$}\}, \quad i \geq 1,
\end{align*}
where $E_j(\mu)$, $j \geq 1$, are i.i.d.\ copies of $E(\mu)$ (the generic service time, distributed exponentially with parameter $\mu$) that are independent from $A$ (the generic inter-arrival time). For $n \geq 1$, $k' \geq 0$, $i \geq k'+1$, $0 \leq k \leq  i-1$, we have %(recall that,to denote an event's complements, we overline the event's notation) 
%(recall that $\overline{E_k}$ denotes the complement of $E_k$)
\begin{equation} \label{eq:21}
\pr\{ Y_{n+1} = i, X_{n+1} \geq  k | X_n = k' \} =   \prod_{j=k'+1}^{i-1} \pr\{ E(j \mu) > A \} \, \pr\{ E_i \cap \overline{E_k}  \},
\end{equation}
where $E(j\mu)$ is distributed exponenitally with parameter $j \mu$ and is independent from $A$, and $\overline{E_k}$ denotes the complement of $E_k$. In particular, for $n \geq 1$,
\begin{subequations}
\begin{gather}
\pr \{ X_n \geq k | Y_n = i \} = \pr\{ \overline{E_k} |  E_i\}, \quad 0 \leq k \leq i-1, \label{eq:23} \\
\pr \{ Y_{n+1}-X_n \geq i | X_n = k \} =
 \prod_{j = k+1}^{k+i - 1} \pr \{ E(j \mu) > A \}, \quad   i \geq 1, \label{eq:22}
\end{gather}
\end{subequations}
(For $n=0$, the formulas \eqref{eq:21} and \eqref{eq:22} read slightly differently, we omit the details.)

Now we introduce prototypes of the fluid limit jump instants $\tau_l$. Recall that, by definition, each of the intervals $(S_{N_l-1}, S_{N_l})$ contains a departure. For $l \geq 1$, let $d_l$ be the last departure instant  within $(S_{N_l-1}, S_{N_l})$ (see Figure~\ref{fig:post_pre_departure_levels}), and put $d_0 := 0$. Along with $d_l$, any departure instant within $(S_{N_l-1}, S_{N_l})$ can be referred to as a prototype of $\tau_l$ because the events (again, see Figure~\ref{fig:post_pre_departure_levels})
\[
D_l := \{  (S_{N_l-1}, S_{N_l}) \text{ contains a single departure instant} \}, \quad l \geq 1,
\]
have a high probability under the fluid scaling, as shown below.

In addition to the events $D_l$ and $E_i$, we also use
%\begin{align*}
%E_i^1 :=\{ \exists!  j=1, \ldots, i \colon E_j(\mu) < A \}, \quad i \geq 1.
%\end{align*}
\begin{align*}
E_i^1 :=\{ E_j(\mu) < A \text{ for a unique $j = 1, \ldots, i$}\}, \quad i \geq 1.
\end{align*}
Note the following difference: the events  $D_l$ relate to departures during  specific inter-arrival periods, and the events  $E_i$ and $E_i^1$ --- to departures from the queue of a specific size.

The proof of Theorem~\ref{th:process_FL} proceeds as follows. Section~\ref{sec:fdd} contains several preparatory statements, of which Lemmas~\ref{lem:joint_MC_FL} and~\ref{lem:unique_service_completion} are the main results. Lemma~\ref{lem:joint_MC_FL} shows convergence of the post-/pre-departure queue levels to the the post-/pre-jump fluid limit levels, and of the departure instants themselves --- to the jump instants of the fluid limit. Lemma~\ref{lem:unique_service_completion} proves that, in the scaling regime,  there is at most a single departure instant per inter-arrival period. Also there are Lemmas~\ref{lem:FL_add_delta}--\ref{lem:FL_add_X} in Section~\ref{sec:fdd}. Lemmas~\ref{lem:FL_add_delta} and \ref{lem:FL_add_X} lay the foundation for the proof of Lemma~\ref{lem:joint_MC_FL} --- they derive the distributional relations~\eqref{eq:FL_MC} between the post- and pre-jump fluid limit levels from the queue-length dynamics. Lemma~\ref{prop:aux_FL} is a preliminary version of Lemma~\ref{lem:unique_service_completion}, it is used in the proof of Lemmas~\ref{lem:FL_add_X} and \ref{lem:unique_service_completion}.

Section~\ref{sec:tightness} contains Lemmas~\ref{lem:fdd}--\ref{lem:tightness} that establish convergence of the scaled queue-length processes to the fluid limit. Lemma~\ref{lem:fdd} shows convergence of the finite-dimensional distributions. Lemma~\ref{lem:tightness} proves relative compactness of the family of the scaled queue-length processes by checking a compact containment and oscillation control properties. Both Lemmas~\ref{lem:fdd} and~\ref{lem:tightness} use the preparatory Lemmas~\ref{lem:joint_MC_FL}--\ref{lem:unique_service_completion} and Remark~\ref{rem:nonexplosive}.

\subsection{Preparatory results} \label{sec:fdd}

First, we prove that the increments $Y_l-X_{l-1}$ weakly converge in the scaling regime to a Weibull distribution. In particular, this implies that the relation  \eqref{eq:12} holds.

\begin{lemma} \label{lem:FL_add_delta} Suppose, for an $n \geq 1$,
\[
\left( X_0, (Y_l, X_l )_{l=1}^{n-1}) \right) / \sqrt{\rho} \Rightarrow \left( \xi_0, (\eta_l, \xi_l)_{l=1}^{n-1} \right)
\]
jointly as $\rho \to \infty$, where the limit vector is absolutely continuous with a density 
\[f_{ \xi_0, (\eta_l, \xi_l)_{l=1}^{n-1}}\left( x_0, (y_l, x_l)_{l=1}^{n-1} \right).\] Then
\[
\left( X_0, (Y_l, X_l)_{l=1}^{n-1}, Y_n-X_{n-1}) \right) / \sqrt{\rho} \Rightarrow \left( \xi_0, (\eta_l, \xi_l )_{l=1}^{n-1}, \Delta_n \right)
\]
jointly, too, and the extended limit vector is absolutely continuous with the density
\begin{equation*}
f_{\xi_0, (\eta_l, \xi_l)_{l=1}^{n-1},  \Delta_n} \left(x_0, (y_l, x_l)_{l=1}^{n-1}, z \right)
= (x_{n-1}+z) e^{-x_{n-1} z - \frac{z^2}{2}} f_{\xi_0, (\eta_l, \xi_0)_{l=1}^{n-1}}\left( x_0, (y_l, x_l)_{l=1}^{n-1}) \right).
\end{equation*}
The lemma also holds for $n \geq 2$ with the components $X_0$, $\xi_0$, $x_0$ omitted throughout the statement.
\end{lemma}

\begin{proof} The approach employed here is to express the distribution of the extended vector in terms of the shorter vector and check the pointwise convergence to the claimed limit distribution. Because of convenience of the more compact notation, we present the proof in the particular case of $n=2$ and the components $X_0$, $\xi_0$, $x_0$ omitted. However, all the derivations translate to the general case in a straightforward  way.

So assume that
$
(Y_1, X_1)/\sqrt{\rho} \Rightarrow (\eta_1, \xi_1)
$
as $\rho \to \infty$, where $(\eta_1, \xi_1)$ is absolutely continuous. We have to check that
$
(Y_1, X_1, Y_2-X_1)/ \sqrt{\rho} \Rightarrow (\eta_1, \xi_1, \Delta_2),
$
where the density of $(\eta_1, \xi_1, \Delta_2)$ is  given by
\begin{equation} \label{eq:6}
f_{\eta_1,\, \xi_1, \, \Delta_2}(y, x, z) =
(x+z) e^{-x z - \frac{z^2}{2}} f_{\eta_1,\, \xi_1}(y, x).
\end{equation}

First, express the distribution of $(Y_1,X_1,Y_2-X_1)$ in terms of  $(Y_1, X_1)$. By \eqref{eq:22} and Proposition~\ref{prop:swap},
\begin{align*}
\pr \{ Y_2-X_1 \geq i | X_1 = k \} &= \prod_{j = k+1}^{k+i - 1} \pr \{ E(j \mu) > A \} = \prod_{j = k+1}^{k+i - 1} \ex e^{-j\mu A} \\
&= \ex e^{-\frac{\sum_{j=k+1}^{k + i - 1} j \wtil A_j}{\rho}} =\ex e^{- \frac{k \wtil S_{i-1}+\sum_{j=1}^{i -1}\wtil S_j}{\rho}}, \quad k \geq 0, \quad i \geq 1.
\end{align*}
Hence, for $y,x,z > 0$,
\begin{align*}
G(\rho) &:=\pr \left\{ \frac{Y_1}{\sqrt{\rho}} \leq y, \frac{X_1}{\sqrt{\rho}} \leq x, \frac{Y_2-X_1}{\sqrt{\rho}} \geq z \right\} \\
&= \sum_{\begin{subarray}{l}
0 \leq l \leq y \sqrt{\rho} , \\
 0 \leq k \leq x \sqrt{\rho}
 \end{subarray}}
 \pr\left\{ \frac{Y_2-X_1}{\sqrt{\rho}} \geq z | X_1 = k \right\} \pr \left\{ Y_1=l, X_1=k \right\} \\
&=\sum_{\begin{subarray}{l}
0 \leq l \leq y \sqrt{\rho}, \\
 0 \leq k \leq x \sqrt{\rho}
 \end{subarray}}
 \ex e^{-\frac{k}{\sqrt{\rho}} \frac{\wtil S_{\lc z \sqrt{\rho} \rc-1}}{\sqrt{\rho}} - \frac{\sum_{j=1}^{\lc z \sqrt{\rho} \rc-1} \wtil S_j}{\rho} } \pr \left\{ Y_1=l, X_1=k \right\} \\
&= \ex e^{-\frac{\wtil S_{\lc z \sqrt{\rho} \rc-1}}{\sqrt{\rho}} \frac{X_1}{\sqrt{\rho}}-\frac{\sum_{j=1}^{\lc z \sqrt{\rho} \rc-1} \wtil S_j}{\rho}}  \ind \left \{ \frac{Y_1}{\sqrt{\rho}} \leq y, \frac{X_1}{\sqrt{\rho}} \leq x \right \}.
\end{align*}
Now take $\rho \to \infty$. We have \[(Y_1,X_1, \wtil{S}_{\lc z \sqrt{\rho} \rc -1})/\sqrt{\rho} \Rightarrow (\eta_1, \xi_1, z)\] and, by Proposition~\ref{prop:lln}(b), $\sum_{j=1}^{\lc z \sqrt{\rho} \rc-1} \wtil S_j / \sqrt{\rho} \Rightarrow z^2/2$, but the expectation in the final expression for $G(\rho)$ is of a discontinuous function of $(Y_1,X_1, \wtil{S}_{\lc z \sqrt{\rho} \rc -1})/\sqrt{\rho}$ and $\sum_{j=1}^{\lc z \sqrt{\rho} \rc-1} \wtil S_j / \sqrt{\rho}$. Still, as shown below, 
\begin{equation} \label{eq:20}
G(\rho) \to \ex e^{-z \xi_1 - \frac{z^2}{2}} \ind\{ \eta \leq y, \xi \leq x \} 
= e^{-z^2/2} \int_{\begin{subarray}{l}
0 <v \leq y , \\
0 <u \leq x
\end{subarray}} e^{-zu} f_{\eta_1, \xi_1}(v,u) \, dv du,
\end{equation}
where differentiation of the right-most side yields the density \eqref{eq:6}.

To prove \eqref{eq:20}, we bound the final expression for $G(\rho)$ from below in two steps. First, for any $\eps> 0$,
\[
G(\rho) \geq \pr \biggl\{ \frac{ \wtil S_{\lc z \sqrt{\rho} \rc-1} } {\sqrt{\rho}}  \leq z+\eps, \frac{\sum_{j=1}^{\lc z \sqrt{\rho} \rc-1} \wtil S_j}{\rho} \leq \frac{z^2}{2}+\eps \biggr\} e^{-\left(\frac{z^2}{2}+\eps\right)}\ex e^{-(z+\eps)\frac{X_1}{\sqrt{\rho}}} \ind \left\{ \frac{Y_1}{\sqrt{\rho}} \leq y, \frac{X_1}{\sqrt{\rho}} \leq x \right\}.
\]  
Second, let $g_-^\eps \colon \nneg^2 \to \nneg$ be a continuous bounded function such that
\[
e^{-(z+\eps)u} \ind \{ v \leq y, u \leq x  \} \geq g_-^\eps(v,u) \geq e^{-(z+\eps)u} \ind \{ v \leq y -\eps, u \leq x-\eps \}, \quad v,u \in \nneg,
\]
then
\[
\liminf_{\rho \to \infty} G(\rho) \geq e^{-\left(\frac{z^2}{2}+\eps\right)} \ex g_-^\eps(\eta_1,\xi_1) \geq e^{-\left(\frac{z^2}{2}+\eps\right)} \ex e^{-(z+\eps)  \xi_1} \ind\{ \eta \leq y-\eps, \xi \leq x-\eps \}.
\]
As we let $\eps \to 0$ in the last inequality, it follows by the absolute continuity of $(\eta_1, \xi_1)$ that
$
\liminf_{\rho \to \infty} G(\rho) \geq e^{-z^2/2}\ex e^{-z \xi_1} \ind\{ \eta \leq y, \xi \leq x \}.
$
Similarly, one can also show the upper bound $\limsup_{\rho \to \infty} G(\rho) \leq e^{-z^2/2} \ex e^{-z \xi_1} \ind\{ \eta \leq y, \xi \leq x \}$.
\end{proof}

The next lemma is an auxiliary result that is used in Lemmas~\ref{lem:FL_add_X} and~\ref{lem:unique_service_completion}. It proves that, if there is a service completion  among $O(\sqrt{\rho})$ customers during an inter-arrival period, then it is the only one with high probability. Note that the probability that there is a service completion at all  is low.
 \begin{lemma} \label{prop:aux_FL}
For any $m > 0$ and $M > m$,
\[
\lim_{\rho \to \infty} \min_{m \sqrt{\rho} \leq i \leq M \sqrt{\rho}} \pr \{ E_i^1 | E_i \} = \lim_{\rho \to \infty} \max_{m \sqrt{\rho} \leq i \leq M \sqrt{\rho}} \pr \{ E_i^1 | E_i \}= 1.
\]
\end{lemma}

\begin{proof}
By symmetry,
\begin{align*}
\pr \{ E_i^1 | E_i \} &= \frac{i \, \pr \{ E_1(\mu) < A, \min_{2 \leq j \leq i} E_j(\mu) > A \}}{\pr \{ \min_{1 \leq j \leq i} E_j(\mu) < A \}} = \frac{i\,  \ex (1 - e^{-\mu A}) e^{-(i-1)\mu A}}{\ex (1 - e^{-i \mu A})} \\
&= \frac{i \, \ex (1 - e^{-\wtil A/\rho}) e^{-(i-1)\wtil A/ \rho}}{\ex (1 - e^{-i \wtil A / \rho})} = \ex \biggl(\frac{1 - e^{-\wtil A / \rho}}{1/\rho} e^{-(i-1)\wtil A / \rho} \biggr) \Big/ \ex \biggl(\frac{1 - e^{-i \wtil A / \rho}}{i / \rho} \biggr).
\end{align*}
Then, by Proposition~\ref{prop:exp_fn}(a),
\[
\pr \{ E_i^1 | E_i\} \geq \ex \biggl( \frac{1 - e^{-\wtil A / \rho}}{1/\rho} e^{-M \wtil A / \sqrt{\rho}} \biggr) \Big/ \ex \biggl( \frac{1 - e^{-m \wtil A / \sqrt{\rho}}}{m / \sqrt{\rho}} \biggr) , \quad m \sqrt{\rho} \leq i \leq M \sqrt{\rho},
\]
and by Proposition~\ref{prop:exp_fn}(b),
$
\liminf_{\rho \to \infty} \min_{m \sqrt{\rho} \leq i \leq M \sqrt{\rho}} \pr \{ E_i^1 | E_i \} \geq 1.
$

Similarly, $\limsup_{\rho \to \infty} \max_{m \sqrt{\rho} \leq i \leq M \sqrt{\rho}} \pr \{ E_i^1 | E_i \} \leq 1$.
\end{proof}

Since, in a queue of size $O(\sqrt{\rho})$, there is at most a single service completion during an inter-arrival period with high probability, the post-departure levels $X_l$ are essentially uniform samples from the corresponding pre-departure levels $Y_l$. This property is inherited by the fluid limit in the form of~\eqref{eq:13}, as implied by the next lemma.

\begin{lemma} \label{lem:FL_add_X} Suppose, for an $n \geq 1$,
\[
\left(X_{l-1}, Y_l \right)_{l=1}^n/ \sqrt{\rho} \Rightarrow (\xi_{l-1}, \eta_l)_{l=1}^n
\]
jointly as $\rho \to \infty$ and the limit vector  is absolutely continuous with a density \[f_{(\xi_{l-1}, \eta_l)_{l=1}^n} \left( (x_{l-1},y_l)_{l=1}^n \right).\] Then the joint weak convergence
\[
\left( (X_{l-1}, Y_l)_{l=1}^n ,X_n \right) / \sqrt{\rho} \Rightarrow \left((\xi_{l-1}, \eta_l)_{l=1}^n, \xi_n \right)
\]
holds as well, where the limit vector is absolutely continuous with the density
\begin{equation} \label{eq:density_add_X}
f_{(\xi_{l-1},\eta_l)_{l=1}^n, \xi_n}\left((x_{l-1}, y_l)_{l=1}^n, x_n \right)
= \  \frac{1}{y_n} f_{(\xi_{i-1},\eta_i)_{i=1}^n}\left((x_{l-1},y_l)_{l=1}^n\right) \, \ind\{ 0 < x_n < y_n \}.
\end{equation}
The lemma also holds with the components $X_0$, $\xi_0$, $x_0$ omitted throughout the statement.
\end{lemma}

\begin{proof} This proof follows the same approach as the proof of  Lemma~\ref{lem:FL_add_delta}, and again, to keep the notation compact, we treat a particular case --- of $n = 1$ and the components $X_0$, $\xi_0$, $x_0$ included. With minor changes, this proof extends to the general case. %The proof below fully illustrates the techniques --- which is sufficient to fully illustrate the techniques used. The other cases Because of the Markov property of the sequence $X_0$, $(Y_l,X_l)$, $l \geq 1$, the proof below adjusts to the general case as the probabilities $\pr \left\{ X_0 = k', Y_1=i \right\}$ are replaced by the probabilities $\pr \{ \left(X_{l-1}, Y_l \right)_{l=1}^n = (k_l,i_l)_{l=1}^n \}$.

By \eqref{eq:23}, for $x'> 0$, $y>x'$ and $x \in (0, y)$,
\begin{align*}
\pr \left\{ \frac{X_0}{\sqrt{\rho}} \leq x', \frac{Y_1}{\sqrt{\rho}} \leq y, \frac{X_1}{\sqrt{\rho}} \geq x \right\} &= \sum_{\begin{subarray}{l}
0 \leq k' \leq x'\sqrt{\rho}, \\
 x \sqrt{\rho}+1 \leq i \leq y \sqrt{\rho}
 \end{subarray}}
 \pr\left\{ \frac{X_1}{\sqrt{\rho}} \geq x | Y_1 = i \right\} \pr \left\{ X_0 = k', Y_1=i \right\} \\
&=\sum_{\begin{subarray}{l}
0 \leq k' \leq x' \sqrt{\rho}, \\
 x \sqrt{\rho}+1 \leq i \leq y \sqrt{\rho}
 \end{subarray}}
\pr\{ \overline{E_{\lc x \sqrt{\rho} \rc}} | E_i\}\pr \left\{ X_0 = k', Y_1=i \right\}.
\end{align*}

Note that, by symmetry, $\pr \{ \overline{E_k} | E_i^1 \} = 1-k/i$, $1 \leq k \leq i$, and hence
\[
\pr \{ \overline{E_k} | E_i \} =  \left( 1-\frac{k}{i} \right) \pr \{ E_i^1 | E_i \} + \pr \{ \overline{E_k} | E_i \cap \overline{E_i^1} \} \pr \{ \overline{E_i^1} | E_i \}.
\]
As we employ the bounds
$
(1-k/i) \pr \{ E_i^1 | E_i \} \leq \pr \{ \overline{E_k} | E_i \} \leq  (1-k/i) + (1-\pr\{ E_i^1 | E_i \}),
$
it follows by Lemma~\ref{prop:aux_FL} that
\begin{align*}
\pr \left\{ \frac{X_0}{\sqrt{\rho}} \leq x', \frac{Y_1}{\sqrt{\rho}} \leq y, \frac{X_1}{\sqrt{\rho}} \geq x \right\} &\sim \sum_{\begin{subarray}{l}
0 \leq k' \leq x'\sqrt{\rho}, \\
 x \sqrt{\rho} \leq i \leq y \sqrt{\rho}
 \end{subarray}} \left( 1- \frac{x \sqrt{\rho}}{i} \right) \pr \left\{ X_0 = k', Y_1=i \right\} \\
 &= \ex \left( 1 - \frac{x}{Y_1/\sqrt{\rho}} \right) \ind \left\{ \frac{X_0}{\sqrt{\rho}} \leq x', x \leq \frac{Y_1}{\sqrt{\rho}} \leq y \right\}.
\end{align*}

Then, following the lines of the part of the proof of Lemma~\ref{lem:FL_add_delta} that estimates the limit of $G(\rho)$, one can show that, by the absolute continuity of $(\xi_0, \eta_1)$, % in the same way as, in the proof of Lemma~\ref{lem:FL_add_delta}, the limit of $G(\rho)$ is estimated, one can show, using the absolute continuity of $(\xi_0, \eta_1)$, that
\begin{align*}
\pr \left\{ \frac{X_0}{\sqrt{\rho}} \leq x', \frac{Y_1}{\sqrt{\rho}} \leq y, \frac{X_1}{\sqrt{\rho}} \geq x \right\} &\to \ex \left( 1 -  \frac{x}{\eta_1} \right) \ind\{ \xi_0 \leq x', x < \eta_1 \leq y \} \\
&= \int_{\begin{subarray}{l}
0 <u \leq x' , \\
x <v \leq y
\end{subarray}} \left(1- \frac{x}{v} \right) f_{\xi_0,\eta_1}(u,v) \, du dv,
\end{align*}
where differentiation of the limit expression produces the density~\eqref{eq:density_add_X} for $n=1$.
\end{proof}

Now, having proved Lemmas~\ref{lem:FL_add_delta} and \ref{lem:FL_add_X}, we are in a position to conclude that the post-/pre-departure queue levels and departure instants weakly converge to the post-/pre-jump levels and jump instants of the claimed fluid limit.
\begin{lemma} \label{lem:joint_MC_FL}
Under the assumption of Theorem~\ref{th:process_FL}, for any $n \geq 1$,
$
 (X_{l-1},Y_l, d_l)_{l=1}^n / \sqrt{\rho} \Rightarrow (\overline{Q}(\tau_{l-1}), \overline{Q}(\tau_l-),\tau_l)_{l=1}^n
$ jointly as $\rho \to \infty$.
\end{lemma}

\begin{proof}
The joint convergence
\begin{equation} \label{eq:24}
\frac{1}{\sqrt{\rho}} (X_{l-1},Y_l)_{l=1}^n \Rightarrow (\overline{Q}(\tau_{l-1}), \overline{Q}(\tau_l-))_{l=1}^n
\end{equation}
 follows by inductive application of Lemmas~\ref{lem:FL_add_delta} and \ref{lem:FL_add_X} to the basis $X_0/\sqrt{\rho} \to \xi_0$ in case $\xi_0$ is absolutely continuous and the basis $Y_1/\sqrt{\rho} \Rightarrow \eta_1$ in case $\xi_0$ is deterministic. In the latter case, $\eta_1$ is absolutely continuous, which can be shown following the lines of the proof of  Lemma~\ref{lem:FL_add_delta}.

Next note that $N_1 = Y_1-X_0+1$ and $N_l-N_{l-1} = Y_l - X_{l-1}$, $l \geq 2$. Since $S_{N_l-1} < d_l < S_{N_l}$, it follows by the continuous mapping theorem that, jointly with~\eqref{eq:24}
\[
\frac{1}{\sqrt{\rho}}(d_l)_{l=1}^n \Rightarrow \frac{1}{\lambda} \left(\sum_{j=1}^l \overline{Q}(\tau_j-) - \overline{Q}(\tau_{j-1}) \right)_{l=1}^n = (\tau_l)_{l=1}^n. \qedhere
\]
\end{proof}

Finally, we show that, each inter-arrival period where departures occur contains a single departure instant only with high probability. This guarantees that, in the limit regime, between the departure instants $d_l$, there are no other departure instants. We use this fact in the proof of convergence to the fluid limit in Section~\ref{sec:tightness}. 

\begin{lemma} \label{lem:unique_service_completion} Under the assumption of Theorem~\ref{th:process_FL}, for any $l \geq 1$, $\pr \{ D_l \} \to 1$ as $\rho \to \infty$.
\end{lemma}

\begin{proof}
Note that $\pr\{ D_l | Y_l = i \} \geq  \pr \{ E_i^1 | E_i \}$, $l \geq 1$, $i \geq 1$. Indeed, if $l \geq 2$, this follows by
\begin{align*}
\pr\{ Y_l = i, D_l \}  &\geq \sum_{k=0}^{i-1} \pr \{ X_{l-1} = k \} \prod_{j = k+1}^{i-1} \pr \{ E(j \mu) > A \} \pr\{ E_i^1 \} , \\
\pr\{ Y_l = i \}  &= \sum_{k=0}^{i-1}  \pr \{ X_{l-1} = k \} \prod_{j = k+1}^{i-1} \pr \{ E(j \mu) > A \} \pr\{ E_i \},
\end{align*}
and by similar formulas if $l=1$.

Then, for any $m > 0$, $M>m$,
$
\pr \{ D_l \} \geq \min_{m \sqrt{\rho} \leq i \leq M \sqrt{\rho}} \pr \{ E_i^1 | E_i \} \pr\{ m \sqrt{\rho} \leq Y_l \leq M \sqrt{\rho} \},
$
and by Lemmas~\ref{prop:aux_FL} and~\ref{lem:joint_MC_FL},
$
\liminf_{\rho \to \infty} \pr \{ D_l \} \geq \pr\{ m < \eta_l < M \},
$
where $\eta_l$ is absolutely continuous. Finally, as $m \to 0$ and $M \to \infty$, the claim of the lemma follows.
\end{proof}

\subsection{Convergence to the fluid limit} \label{sec:tightness}
In this section, we prove that the family of the scaled queue-length processes $\overline{Q}^\rho(\cdot)$ is relatively compact (in Lemma~\ref{lem:tightness}), and that their finite-dimensional distributions converge to those of the claimed limit process $\overline{Q}(\cdot)$ (in Lemma~\ref{lem:fdd}). These are the ingredients of weak convergence in the Skorokhod space $\mathbf{D}$.

\begin{lemma} \label{lem:fdd}
Under the assumption of Theorem~\ref{th:process_FL}, for any $n \geq 1$ and partition $0 \leq t_1 < t_2 < \ldots < t_n$,
$
\left( \overline{Q}^\rho(t_j) \right)_{j=1}^n \Rightarrow \left( \overline{Q}(t_j) \right)_{j=1}^n
$
jointly as $\rho \to \infty$.
\end{lemma}

\begin{lemma} \label{lem:tightness}
Under the assumption of Theorem~\ref{th:process_FL}, the family of the processes $\overline{Q}^\rho(\cdot)$, $\rho > 0$, is relatively compact.
\end{lemma}

\begin{proof}[Proof of Theorem~\ref{th:process_FL}] 
It follows by Lemmas~\ref{lem:fdd} and~\ref{lem:tightness} and \cite[Theorem 3.7.8(b)]{EthierKurtz} that the processes $\overline{Q}^\rho(\cdot)$ converge weakly in $\mathbf{D}$ to $\overline{Q}(\cdot)$.
\end{proof}

%Here we prove that the family of the scaled queue-length processes is relatively compact (in Lemma~\ref{lem:tightness}), and that their finite-dimensional distributions converge to those of the claimed fluid limit (in Lemma~\ref{lem:fdd}). Then it follows by \cite[Theorem 3.7.8(b)]{EthierKurtz} that the scaled queue-length processes converge weakly to the claimed fluid limit in the Skorokhod space $\mathbf{D}$.

In the proofs of Lemmas~\ref{lem:fdd} and~\ref{lem:tightness} below,  $A(\cdot)$ denotes the arrival process. We also use its fluid-scaled version
\[
\overline{A}^\rho (t) := \frac{A(\sqrt{\rho} t)}{\sqrt{\rho}}, \quad t \geq 0.
\]

\begin{proof}[Proof of Lemma~\ref{lem:fdd}] \begin{comment}The main idea of the proof is that the queue size at an arbitrary point in time is defined by how long ago there has been a departure and what was the size of the queue back then.

 to express the queue size at an arbitrary point in time in terms of  the time that has elapsed since the last departure instant $d_l$, the queue-level $X_l$ at that departure instant, and the arrival process $A(\cdot)$, post-departure queue levels $X_l$ $d_l$, and the arrival process.  introduced in the beginning of the section, whose scaling limits are known. \end{comment}
In case the limit initial condition $\overline{Q}(0)$ is deterministic,  assume without loss of generality that $t_1 > 0$. We prove the weak convergence by verifying the pointwise convergence of the distribution function.

The main idea is to express the queue size at an arbitrary point in time in terms of the departure instants $d_l$, post-departure queue levels $X_l$ and the arrival process, whose scaling limits are known. We achieve this by tracing back to the last departure instant $d_l$ and making sure that there has not been any other departures since then.

More formally, fix $x_j > 0$, $1 \leq j \leq n$. By specifying the last departure instants $d_l/\sqrt{\rho}$ preceding the $t_j$'s, we get
\[
\pr \{\overline{Q}^\rho(t_j) \leq x_j, 1 \leq j \leq n \} =
\sum_{0 \leq l_1 \leq l_2 \leq \ldots \leq l_n} \pr \{ B_{(l_j)_{j=1}^n}\},
\]
where
\[
B_{(l_j)_{j=1}^n} := \left\{ \frac{d_{l_j}}{\sqrt{\rho}} \leq t_j, \frac{d_{l_j+1}}{\sqrt{\rho}} > t_j, \overline{Q}^\rho(t_j) \leq x_j, 1 \leq j \leq n \right\}.
\]
On the event $B_{(l_j)_{j=1}^n}$, if there are no other departures between $d_{l_j}$ and $\sqrt{\rho} t_j$, then $\overline{Q}^\rho(t_j) = X_{l_j}/\sqrt{\rho} + \overline{A}^\rho(t_j) - A(d_{l_j})/\sqrt{\rho}$, and hence, for any $K \geq 1$,
\[
B_{(l_j)_{j=1}^n} \bigcap_{l=1}^K D_l
= B^1_{(l_j)_{j=1}^n} \bigcap_{l=1}^K D_l, \quad 0 \leq l_1 \leq l_2 \leq \ldots \leq l_n < K,
\]
where
\[
B^1_{(l_j)_{j=1}^n} := \left\{ \frac{d_{l_j}}{\sqrt{\rho}} \leq t_j, \frac{d_{l_j+1}}{\sqrt{\rho}} > t_j, \frac{X_{l_j}}{\sqrt{\rho}} + \overline{A}^\rho(t_j) - \frac{A(d_{l_j})}{\sqrt{\rho}} \leq x_j, 1 \leq j \leq n \right\}.
\]
\begin{comment}
 and consider the events
\begin{align*}
B_{(l_j)_{j=1}^n} &:= \left\{ \frac{d_{l_j}}{\sqrt{\rho}} \leq t_j, \frac{d_{l_j+1}}{\sqrt{\rho}} > t_j, \overline{Q}^\rho(t_j) \leq x_j, 1 \leq j \leq n \right\}, \\
B^1_{(l_j)_{j=1}^n} &:= \left\{ \frac{d_{l_j}}{\sqrt{\rho}} \leq t_j, \frac{d_{l_j+1}}{\sqrt{\rho}} > t_j, \frac{X_{l_j}}{\sqrt{\rho}} + \overline{A}^\rho(t_j) - \frac{A(d_{l_j})}{\sqrt{\rho}} \leq x_j, 1 \leq j \leq n \right\},
\end{align*}
 where we specify the last departure instants $d_l$ preceding the $t_j$'s, and in $B^1_{(l_j)_{j=1}^n}$, we replace the queue size by the expression that is valid if there are no other departures between the $d_{l_j}$ and the corresponding $t_j$'s. These events satisfy
\begin{gather*}
\pr \{\overline{Q}^\rho(t_j) \leq x_j, 1 \leq j \leq n \} =
\sum_{0 \leq l_1 \leq l_2 \leq \ldots \leq l_n} \pr \{ B_{(l_j)_{j=1}^n}\}, \\
B_{(l_j)_{j=1}^n} \bigcap_{l=1}^K D_l
= B^1_{(l_j)_{j=1}^n} \bigcap_{l=1}^K D_l, \quad 0 \leq l_1 \leq l_2 \leq \ldots \leq l_n < K.
\end{gather*}
\end{comment}
The above relations imply the bounds (recall that $\overline{D_l}$ denotes the complement of the event $D_l$)
\begin{gather*}
\sum_{0 \leq l_1 \leq l_2 \leq \ldots \leq l_n < K} \pr \{ B^1_{(l_j)_{j=1}^n}\} - \sum_{l=1}^K \pr \{ \overline{D_l} \}
\leq \pr \{ \overline{Q}^\rho(t_j) \leq x_j, 1 \leq j \leq n \} \\
\leq  \sum_{0 \leq l_1 \leq l_2 \leq \ldots \leq l_n < K} \pr \{ B^1_{(l_j)_{j=1}^n}\}  + \sum_{l=1}^K \pr \{  \overline{D_l} \} + \pr \left\{ \frac{d_K}{\sqrt{\rho}} \leq t_n \right\},
\end{gather*}
where we keep $K$ fixed for the moment and let $\rho \to \infty$. Note that, by the functional law of large numbers,
$\left(\overline{A}^\rho (t_j) - A(d_{l_j})/\sqrt{\rho} \right)_{j=1}^n \Rightarrow\lmb \left( t_j - \tau_{l_j}\right)_{j=1}^k$ jointly with the weak convergence in Lemma~\ref{lem:joint_MC_FL}. Then, by the absolute continuity of the weak limits in Lemma~\ref{lem:joint_MC_FL}, we have
\begin{align*}
\liminf_{\rho \to \infty} \pr \{ \overline{Q}^\rho(t_j) \leq x_j, 1 \leq j \leq n \} &\geq \sum_{0 \leq l_1 \leq l_2 \leq \ldots \leq l_n < K} \pr \{ B^\ast_{(l_j)_{j=1}^n} \}
 \\
\limsup_{\rho \to \infty} \pr \{ \overline{Q}^\rho(t_j) \leq x_j, 1 \leq j \leq n \}
&\leq \sum_{0 \leq l_1 \leq l_2 \leq \ldots \leq l_n < K} \pr \{ B^\ast_{(l_j)_{j=1}^n} \}    + \pr \left\{ \tau_K \leq t_n \right\},
\end{align*}
where the events
\[
B^\ast_{(l_j)_{j=1}^n} := \{ \tau_{l_j} \leq t_j, \tau_{l_j+1} > t_j, \overline{Q}(\tau_{l_j}) + \lambda (t_j - \tau_{l_j}) \leq x_j, 1 \leq j \leq n \}
\]
are analogues of the events $B^1_{(l_j)_{j=1}^n}$ for the fluid limit. Finally, we let $K \to \infty$, and then  Remark~\ref{rem:nonexplosive} implies that
\[
\pr \{ \overline{Q}^\rho(t_j) \leq x_j, 1 \leq j \leq n \} \to \sum_{0 \leq l_1 \leq l_2 \leq \ldots \leq l_n} \pr \{ B^\ast_{(l_j)_{j=1}^n} \} =  \pr \left\{ \overline{Q}(t_j) \leq x_j, 1 \leq j \leq n \right\}. \qedhere
\]
\end{proof}

To prove the relative compactness of the processes $\overline{Q}^\rho(\cdot)$ in $\mathbf{D}$, we employ the standard tools of compact containment and oscillation control.

\begin{proof}[Proof of Lemma~\ref{lem:tightness}]
This proof makes use of the following moduli of continuity: for any $x \in \mathbf{D}$ and $\dlt, T > 0$,
\begin{comment}\[
\omega(x, \dlt, T) := \sup_{\begin{subarray}{l}s,t \in [0,T], \\ |s-t| \leq \dlt \end{subarray}} |x(s)-x(t)|
\] \end{comment}
\[
\omega(x, \dlt, T) := \sup_{s,t \in [0,T], |s-t| \leq \dlt} |x(s)-x(t)|
\]
and
\[
\omega'(x, \delta, T) := \inf_{(t_i)_{i=1}^k} \max_{1 \leq i \leq k} \sup_{s, t \in [t_{i-1},t_i)} |x(s)-x(t)|,
\]
where the infimum is taken over all partitions $0=t_0<t_1< \ldots < t_{n-1} < T \leq t_n$ such that $\min_{1 \leq i \leq n} (t_i-t_{i-1}) > \delta$

By \cite[Corollary 3.7.4]{EthierKurtz}, in order to prove the lemma, it suffices to check that, for any $\eps, T > 0$, there exist $M, \delta > 0$ such that
\begin{subequations}
\begin{align}
\label{eq:comp_contain}
\liminf_{\rho \to \infty} &\pr \{ \overline{Q}^\rho(T) \leq M \} \geq 1-\eps, \\
\label{eq:osc_control}
\limsup_{\rho \to \infty} &\pr \{ \omega'(\overline{Q}^\rho, \delta, T) \geq \eps \} \leq \eps.
\end{align}
\end{subequations}

So fix $\eps, T > 0$. The compact containment property \eqref{eq:comp_contain} follows as we upper-bound the queue size by arrivals,
$
\overline{Q}^\rho(T) \leq \overline{Q}^\rho(0) + \overline{A}^\rho(T) \Rightarrow \xi_0 + \lambda T
$
as $\rho \to \infty$.

As for the proof of the oscillation control property  \eqref{eq:osc_control},  the key idea is to construct a partition that includes all departure instants up to $T$. This allows to bound oscillations of the scaled queue length $\overline{Q}^\rho(\cdot)$ by oscillations of the scaled arrival process $\overline{A}^\rho(\cdot)$, and the latter can be made arbitrarily small by the convergence of  $\overline{A}^\rho(\cdot)$ to a continuous limit as $\rho \to \infty$.

To construct the partition of interest (an example is  given in  Figure~\ref{fig:partition}), for each $K \geq 1$ and $\dlt >0$, introduce the event
\[
B_{K,\delta} := \left\{\frac{d_K}{\sqrt{\rho}} \geq T,  \frac{d_l-d_{l-1}}{\sqrt{\rho}} \geq 2\delta, 1 \leq l \leq K  \right\} \bigcap_{l=1}^K D_l,
\]
and on this event, define the indices
\[
i_0:= 0,  \quad i_l: = i_{l-1} + \left\lfloor \frac{d_l - d_{l-1}}{2\delta \sqrt{\rho}} \right\rfloor, \quad 1 \leq l \leq K,
\]
and the partition
\[
t'_{i_l+j} := d_j / \sqrt{\rho}  + 2\delta l, \quad 0 \leq l \leq K-1, \quad 0 \leq j \leq i_l-i_{l-1}-1,  \quad t'_{i_K} = d_K/\sqrt{\rho},
\]
where the partition elements that are relevant for the rest of the proof are those with indices up to
\[
n' := \min\{i = 0, 1, \ldots, i_K \colon  t'_i \geq T  \}.
\]

\begin{figure}%[!b] 
\centering
\includegraphics[scale=0.75]{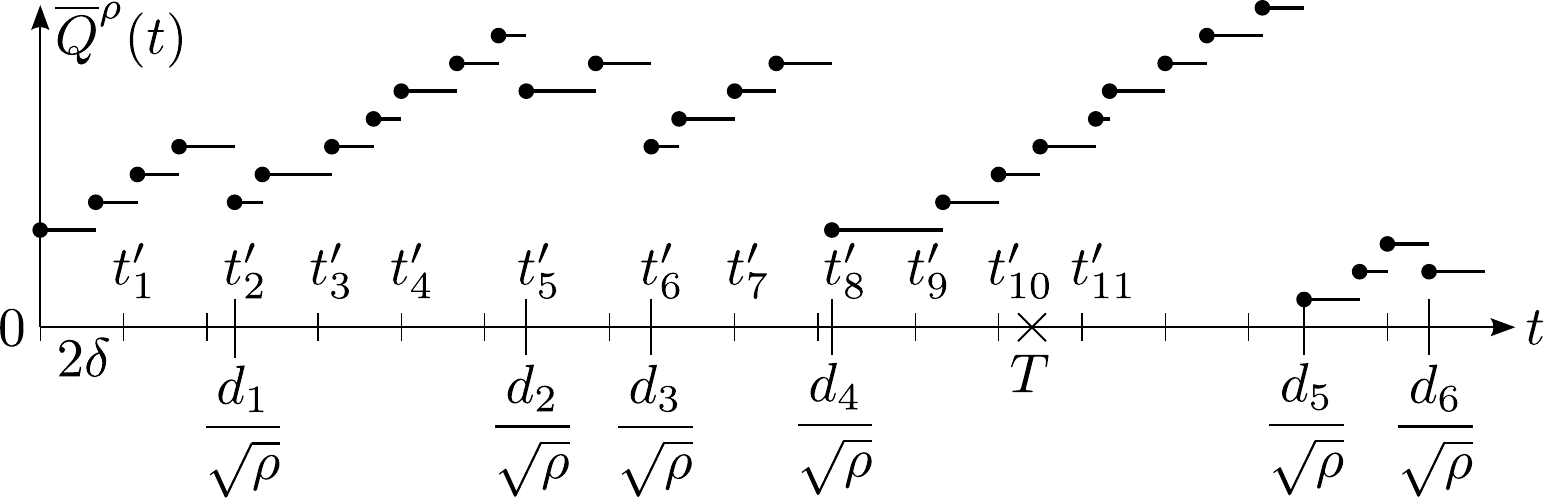}
\caption{A realization from the event $B_{6,\delta}$ (i.e., $K=6$). Here $i_1 = 2$, $i_3 = 5$, $i_3 = 6$, $i_4 = 8$, $i_5 = 13$, $i_6 = 14$, and the relevant partition elements are $(t'_i)_{i=0}^{11}$ (i.e., $n' = 11$).}
\label{fig:partition}
\end{figure}

Note that the constructed partition has the following properties: (a) there are no departures between the elements of $(t'_i)_{i=1}^{n'}$, (b) $t'_{n'} \leq T + 2 \dlt$, and (c) $2 \dlt \leq t'_i-t'_{i-1} \leq 4 \dlt$, $1\leq i \leq n'$. Due to  these properties, oscillations of  $\overline{Q}^\rho(\cdot)$ admit the following upper bound on the event $B_{K,\delta}$:
\begin{align*}
\omega'(\overline{Q}^\rho, \delta, T) &\leq \max_{1 \leq i \leq n'} \sup_{s, t \in [t'_{i-1},t'_i)} |\overline{Q}^\rho(s)-\overline{Q}^\rho(t)| \\
&\leq \max_{1 \leq i \leq n'}  \overline{A}^\rho(t'_i)-\overline{A}^\rho(t'_{i-1}) \leq \omega(\overline{A}^\rho, 4\delta, T+2 \delta).
\end{align*}
Thus, we have
\begin{align*}
\pr \{ \omega'(\overline{Q}^\rho, \delta, T) \geq \eps \} \leq&  \, \pr \{ \omega(\overline{A}^\rho, 4\delta, T+2\dlt) \geq \eps \}  \nonumber \\
&+ \pr \left\{\frac{d_K}{\sqrt{\rho}} < T \right\} + \sum_{l=1}^K \pr \left\{ \frac{d_l-d_{l-1}}{\sqrt{\rho}} < 2\delta \right\} + \sum_{l=1}^K \pr \{ \overline{D_l} \}. %\label{eq:9}
\end{align*}
Now, since $\overline{A}^\rho(\cdot) \Rightarrow \Lambda (\cdot)$ in $\mathbf{D}$ as $\rho \to \infty$, where $\Lambda(t) := \lmb t$, $t \geq 0$, and since, for any $\dlt', T' > 0$, the mapping $\omega(\cdot, \dlt', T') \colon \mathbf{D} \to [0,\infty)$ is continuous at any continuous $x \in \mathbf{D}$, it follows that
 for any $\dlt > 0$,
\begin{equation*} %\label{eq:8}
\omega(\overline{A}^\rho, 4\delta, T+2\dlt) \Rightarrow \omega(\Lambda, 4 \dlt, T + 2 \dlt) = 4 \dlt \lmb \quad \text{as $\rho \to \infty$}.
\end{equation*}
Together with Lemmas~\ref{lem:joint_MC_FL} and~\ref{lem:unique_service_completion}, the last two displays imply that,
for any any $\dlt \in (0, \eps/(4 \lmb))$ and  $K \geq 1$,
\[
\limsup_{\rho \to \infty} \pr \{ \omega'(\overline{Q}^\rho, \delta, T) \geq \eps \} \leq  \pr \left\{\tau_K \leq T \right\} + \sum_{l=1}^K \pr \{ \tau_l-\tau_{l-1} \leq 2\delta \}.
\]
Finally, by Remark~\ref{rem:nonexplosive}, one can pick a sufficiently large $K$ such that
$
\pr \left\{\tau_K \leq T \right\} \leq \eps/2,
$
and then pick a sufficiently small $\dlt \in (0, \eps/(4 \lmb)$ such that
$
\sum_{l=1}^K \pr \{ \tau_l-\tau_{l-1} \leq 2\delta \} \leq \eps/2,
$
and hence the oscillation control property \eqref{eq:osc_control} indeed holds.
\end{proof}

\section*{Appendix}
\appendix
Here we provide the proofs of Lemmas~\ref{lem:st_dn_moments} and~\ref{lem:tr_dn}, which assume Poisson arrivals.

\begin{proof}[Proof of Lemma~\ref{lem:st_dn_moments}] This proof makes use of the generating function \[
\pi(z) := \sum_{k \geq 0} \pi_k z^k, \quad |z| \leq 1.
\]
Since the arrival process is Poisson, the queue-length process $Q(\cdot)$ forms a continuous-time Markov chain with transition rates
\begin{align*} %\label{eq:ctmc_rates}
\begin{split}
q(i, i+1) &= \lambda, \quad i \geq 0, \\
q(i,j) &= \mu, \quad i \geq 1, \  0 \leq j \leq i-1.
\end{split}
\end{align*}
The balance equations for $Q(\cdot)$ and sets $\{0, 1, \ldots, k\}$ read as
\begin{equation*} %\label{eq:balance}
\lmb \pi_k = (k+1) \mu \sum_{i \geq k+1} \pi_i, \quad k \geq 0,
\end{equation*}
and imply the following relation
\[
\rho \pi(z) = \sum_{k \geq 0} z^k (k+1) \sum_{i \geq k+1} \pi_i = \sum_{i \geq 1} \pi_i \sum_{k=0}^{i-1} (k+1) z^k
=  r'(z),
\]
where
\[
r(z) := \sum_{i \geq 0} \pi_i \sum_{k=0}^i z^k.
\]

Note that, by the Taylor expansion around $z=1$, the generating function $\pi(z)$ admits the alternative representation
\begin{equation} \label{eq:taylor1}
\pi(z) = \sum_{n \geq 0} \frac{m_n}{n!} (z-1)^n.
\end{equation}
We now derive the Taylor expansion around $z=1$ for $r'(z)$ as well. We have
\begin{align*}
r(z)
 &= \sum_{i \geq 0} \pi_i \left( z^i + \frac{z^i-1}{z-1} \right) = \pi(z) + \frac{\pi(z)-1}{z-1} \\
 &= \sum_{n \geq 0} \frac{m_n}{n!} (z-1)^n + \sum_{ n\geq 1} \frac{m_n}{n!} (z-1)^{n-1},
\end{align*}
and hence
\begin{equation} \label{eq:taylor2}
r'(z) = \sum_{n \geq 0} \left( \frac{m_{n+1}}{n!} + \frac{m_{n+2}}{(n+2)!} (n+1) \right)  (z-1)^n.
\end{equation}
Finally, by equating the coefficients in the Taylor series \eqref{eq:taylor1} multiplied by $\rho$ and \eqref{eq:taylor2}, we obtain the relation \eqref{eq:moments_recursion}.
\end{proof}

\begin{proof}[Proof of Lemma~\ref{lem:tr_dn}]
With the notation
\[
p_k(t) := \pr \{ Q(t) = k \}, \quad k \geq 0,
\]
The Kolmogorov equations for the Markov process $Q(\cdot)$ read as follows:
\[
\begin{cases}
p_0'(t) &= -\lmb p_0(t) + \mu \sum_{l \geq 1} p_l(t), \\
p_k'(t) &= -(\lmb+k\mu) p_k(t) + \lmb p_{k-1}(t) + \mu \displaystyle{\sum_{l \geq k+1}} p_l(t), \quad k \geq 1.
\end{cases}
\]
By the identity $\sum_{l \geq k+1} p_l(t)  = 1 - \sum_{l =0}^k p_l(t)$, the above system can be rewritten as
\begin{equation} \label{eq:kolm_original}
\begin{cases}
p_0'(t) + (\lmb+\mu) p_0(t) &= \mu, \\
p_k'(t) + (\lmb+(k+1)\mu) p_k(t) &= \mu - \mu \displaystyle{\sum_{l=0}^{k-1}} p_l(t)  + \lmb p_{k-1}(t), \quad k \geq 1.
\end{cases}
\end{equation}
We now multiply the $k$-th equation  by $e^{(\lmb+(k+1)\mu)t}$ and make the substitution
\[
u_k(t) := e^{(\lmb+(k+1)\mu)t} p_k(t) , \quad k \geq 0,
\]
which yields the new system
\begin{equation} \label{eq:kolm_u}
\begin{cases}
u_0'(t) &= \mu e^{(\lmb+\mu)t},
\\
u_k'(t) &= \mu e^{(\lmb+(k+1)\mu)t} - \mu \displaystyle{ \sum_{l=0}^{k-1}} e^{(k-l)\mu t} u_l(t) + \lmb e^{\mu t} u_{k-1}(t), \quad k \geq 1.
\end{cases}
\end{equation}
It can be shown by induction that any solution to  \eqref{eq:kolm_u} is of the form
\begin{equation*}
\begin{cases}
u_0(t) &= \widetilde{C}_{0,1} e^{(\lmb+\mu)t} + \widetilde{C}_{0,0}, \\
u_k(t) &= \widetilde{C}_{k,k+1} e^{(\lmb+(k+1)\mu)t} + \displaystyle{\sum_{l=1}^k} \widetilde{C}_{k,l} e^{l \mu t} + \widetilde{C}_{k,0}, \quad k \geq 1,
\end{cases}
\end{equation*}
with some constants  $\widetilde{C}_{k,l}$. Hence, any solution to the original system \eqref{eq:kolm_original} is of the form~\eqref{eq:kolm_sol}.

It is left to check that the constants $C_{k,i}$ in \eqref{eq:kolm_sol} satisfy the recursive relations \eqref{eq:kolm_induction}.  The identity \eqref{eq:kolm_const} is the limit of \eqref{eq:kolm_sol} as $t \to \infty$. The relation \eqref{eq:kolm:middle} follows as one plugs \eqref{eq:kolm_sol} back into the system \eqref{eq:kolm_original} and equates the corresponding coefficients of the (linearly independent) functions $e^{-(\lmb+i\mu)t}$, $1 \leq i \leq k$. Finally, the relation \eqref{eq:kolm_last} follows by putting $t=0$ in \eqref{eq:kolm_sol}.
\end{proof}

\vspace{2mm}

{\small \noindent {\bf Acknowledgments.} This research was partly funded by the NWO Gravitation Project N{\sc etworks}, Grant Number 024.002.003.  The authors thank F. Simatos, P. Taylor, and B. Zwart for helpful discussions. }

%\bibitem{MauZwart} Maulik, Zwart
\end{document}